\documentclass[12pt,reqno,letterpaper]{amsart}

\usepackage{tikz}
\usetikzlibrary{decorations.markings}

\usepackage{amsmath,amssymb,amsthm}
\usepackage[left=2cm,top=2cm,right=2cm,bottom=2cm]{geometry}
\usepackage[normalem]{ulem}
\usepackage{mathtools}
\usepackage{enumerate}
\usepackage[shortlabels]{enumitem}
\usepackage{xspace}
\usepackage{siunitx}
\usepackage[a-1b]{pdfx} 
\usepackage{cleveref}

\let\ab\allowbreak

\newcommand{\eps}{\epsilon}

\numberwithin{equation}{section}

\newtheorem{theorem}{Theorem}
\newtheorem{lemma}[theorem]{Lemma}

\newtheorem{claim}[theorem]{Claim}

\theoremstyle{definition}

\newcommand{\set}[1]{\left\{#1\right\}}
\newcommand{\parens}[1]{\left( #1 \right)}
\newcommand{\floor}[1]{\left\lfloor #1 \right\rfloor}

\def\E{\mathbb{E}}

\newcommand{\Po}{\operatorname{Po}}
\newcommand{\Bin}{\operatorname{Bin}}

\newcommand{\bb}{\mathbf{b}}
\renewcommand{\aa}{\mathbf{a}}
\newcommand{\abs}[1]{\left|{#1}\right|}
\newcommand{\card}[1]{\left|{#1}\right|}
\newcommand{\tcard}[1]{|{#1}|}

\renewcommand{\flat}[1]{\dot{#1}}
\newcommand{\Xflat}{\flat{X}}

\newcommand{\flatVD}{\flat{V}(D)}
\newcommand{\vol}{\operatorname{vol}}

\newcommand{\Np}{N^+}
\newcommand{\Nm}{N^-}

\def\d{\delta}
\def\D{\Delta}

\def\la{\lambda}

\def\to{\rightarrow}

\def\calr\ensuremath
\def\calb{\ensuremath{\mathcal B}}

\def\bk{\setminus}

\newcommand{\ignore}[1]{\relax}
\DeclareMathSymbol{\Nu}{\mathalpha}{operators}{"4E}
\renewcommand{\Nu}{n_H}
\renewcommand{\nu}{\gapd}
\newcommand{\ax}{a_x}
\newcommand{\bx}{b_x}

\newcommand{\ttfrac}[2]{{#1}/{#2}}
\newcommand{\rx}{\rho}
\newcommand{\one}[1]{{\mathbf 1} \! \left[ {#1} \right] }
\newcommand{\onex}{\mathbf 1}

\newcommand{\calP}{\ensuremath{\mathcal{P}}\xspace}
\newcommand{\calC}{\ensuremath{\mathcal{C}}\xspace}
\def\calM{\ensuremath{\mathcal M}\xspace}
\def\calA{\ensuremath{\mathcal A}\xspace}
\def\calB{\ensuremath{\mathcal B}\xspace}
\def\calD{\ensuremath{\mathcal D}\xspace}
\def\calH{\ensuremath{\mathcal H}\xspace}

\newcommand{\gapd}{\eps_D}
\newcommand{\gapq}{\eps_Q}
\newcommand{\dd}{\delta_D}
\newcommand{\AX}{A_X}
\newcommand{\BX}{B_X}

\newcommand{\nn}{N}

\newcommand{\etad}{\eta_{c,k}(\d)}

\title[Matchings and loose cycles in the semirandom hypergraph model]{Perfect matchings and loose Hamilton cycles \\ in the semirandom hypergraph model}

\author{Michael Molloy}
\address{Department of Computer Science,
University of Toronto, Toronto, ON, Canada}
\email{molloy@cs.toronto.edu}

\author{Pawe\l{} Pra\l{}at}
\address{Department of Mathematics, Toronto Metropolitan University, Toronto, Canada}
\email{pralat@torontomu.ca}

\author{Gregory B.\ Sorkin}
\address{Department of Mathematics,
The London School of Economics and Political Science, London, England}
\email{g.b.sorkin@lse.ac.uk}

\begin{document}

\date{24 September 2024}

\begin{abstract}
We study the 2-offer semirandom 3-uniform hypergraph model on $n$ vertices.
At each step, we are presented with 2 uniformly random vertices.
We choose any other vertex, thus creating a hyperedge of size 3.
We show a strategy that constructs a perfect matching, and another that constructs a loose Hamilton cycle, both succeeding
asymptotically almost surely
within $\Theta(n)$ steps.
Both results extend to $s$-uniform hypergraphs.
The challenges with hypergraphs, and our methods, are qualitatively different from what has been seen for semirandom graphs. 
Much of our analysis is done on an auxiliary graph
that is a uniform $k$-out subgraph of a random bipartite graph,
and this tool may be useful in other contexts.
\end{abstract}

\maketitle

\section{Introduction and Main Results}

\subsection{The model} 
The canonical Erd\H{o}s-R\'enyi random graph model~\cite{ER1959} has been studied exhaustively since the 1960s.
It is deeply understood, 
and by now a great deal of attention has been devoted to variants and special cases, including hypergraphs, sparse graphs, planar graphs, and maker-breaker games. 
A version of the Erd\H{o}s-R\'enyi model is that, given $n$ vertices, edges arrive, one per round, uniformly at random.
One variant, suggested around 2000 by Achlioptas and first analyzed in~\cite{BohmanFrieze2001}, is that edges arrive two at a time, and just one of the two is added to the graph.

The \textbf{semirandom graph process} we consider here has a similar character. 
Instead of being offered a pair of edges, we are offered a vertex, and we may add any edge on that vertex; another way to view it is that we are offered a star, and choose one of its edges to add.
This process was suggested by Peleg Michaeli, introduced formally in~\cite{beneliezer2019semirandom} in 2020, and has already received a good deal of attention.
It can be viewed as a ``one player game'', and generalizes several well-known and important processes including the Erd\H{o}s-R\'enyi random graph process itself, and the  ``$k$-out process'' (about which we will say more).
We concentrate on its natural generalization to hypergraphs, as proposed in~\cite{behague2022}.

In the \textbf{$r$-offer $s$-uniform semirandom model}, the player constructs a sequence of $s$-uniform hypergraphs $H_t$ on vertex set $[n]:=\{1,\ldots, n\}$, with the goal that $H_t$ satisfies a given property $\mathcal{P}$ as quickly as possible.
In each \textbf{round} $t \in \mathbb N$, a uniformly random set of $r$ vertices is \textbf{offered}, and the player \textbf{chooses} $s-r$ additional vertices according to some \textbf{strategy}, to make a hyperedge that is added to $H_{t-1}$ to form $H_t$.
When choosing the $s-r$ vertices, the player has full knowledge of the $r$ vertices offered and of the past, including the hypergraph $H_{t-1}$ at the end of the previous round.

\subsection{Our contribution} \label{sec:contribution}
Our contribution is to resolve some natural first questions for the hypergraph semirandom process, notably in the case where each offer specifies two vertices rather than just one.
We believe that this is the first work on semirandom hypergraphs that is qualitatively different from anything that has been done for semirandom graphs.
In this work, the property \calP is that the semirandom hypergraph should contain as a subgraph a desired structure, namely a perfect matching or a loose Hamilton cycle.
Constructing such spanning structures has also been one focus of the study of semirandom graphs.

A main point of difference is that for graphs, it has always been possible to quickly and simply construct most of the structure desired, then cleverly fill in the final gaps.
That approach does not seem to work here.
Instead, we first generate a large absorbing structure of some sort,
build a small part of the desired structure using the few vertices absent from the absorber,
then use the absorber to complete the desired structure.
Also, while the semirandom work on graphs has been able to address the desired graph directly, we found it necessary to work in an auxiliary graph. 

To suggest the profoundly greater challenges of semirandom hypergraphs compared with graphs, let us mention an approach to matching that works for graphs but seemingly not for hypergraphs.
It is easy to semirandomly construct a random $k$-out graph or hypergraph; we discuss this further in \Cref{background}.
In the graph case, by~\cite{frieze1986maximum}, the $k$-out almost surely contains a perfect matching.
In the hypergraph case, by contrast,
Devlin and Kahn~\cite{devlin2017perfect}
observe that the corresponding statement is ``almost certainly correct [but] likely to be difficult [to show]'',
as it is stronger than the 2008 resolution of Shamir's Problem by~\cite{JKV2008}
and would imply
a (still open) ``natural guess regarding a beautiful problem of Frieze and Sorkin'' (from~\cite{frieze2015efficient}).
Devlin and Kahn make progress towards this by proving the presence of a perfect \emph{fractional} matching, but that is not sufficient for our purpose.
Even if this approach did work, it would likely be existential;
we prefer (and obtain) an efficient construction.

To avoid confusion, we remark that our successful approach to hypergraph matching has some superficial similarity to this. 
We are able to construct a $k$-out auxiliary graph (not hypergraph) on most (unfortunately not all) of the vertices to give an absorbing structure that guarantees a hypergraph matching after the exceptional vertices are dealt with.

\subsection{Our results}

Results presented in this paper are asymptotic by nature. We say that a property $\mathcal P$ holds \textbf{asymptotically almost surely} (\textbf{a.a.s.}) if the probability that
$H$ has property $\mathcal P$ tends to $1$ as $n$ goes to infinity.

\smallskip

In our first main result, we concentrate on the 2-offer semirandom construction of a 3-uniform hypergraph $H$ on $n$ vertices, $n$ divisible by 3, so that $H$ will contain a perfect matching (a partition of the set of vertices into $n/3$ hyperedges).

\begin{theorem} \label{mainthm}
For some constant $C$, the 2-offer 3-uniform semirandom model has a strategy that,
for $n$ divisible by $3$,
in time $t=Cn$, gives a hypergraph $H_t$ a.a.s.\ containing a perfect matching.
\end{theorem}

We present the strategy explicitly, and show how to construct a matching in $H$ (not merely proving its existence).
This result is clearly best possible up to the constant $C$, which we have made no effort to optimise.

Our linear-time construction extends to perfect matchings in an $s$-uniform hypergraph when the player is offered $r=2$ vertices and is able to select $s-2$ vertices. Moreover, it immediately implies the same result for a ``1-offer'' semirandom model, as the player can simply simulate a 2nd random offer vertex with the first of the $s-1$ chosen vertices.

\begin{theorem} \label{s-uniform}
For any $s \ge 3$ and $r \in \set{1,2}$,
for some constant $C$, the $r$-offer $s$-uniform semirandom model has a strategy that,
for $n$ divisible by $s$,
in time $t=Cn$, gives a hypergraph $H_t$  a.a.s.\ containing a perfect matching.
\end{theorem}

Our second main result is a somewhat similar strategy to construct a 3-uniform hypergraph $H$ on $n$ vertices, $n$ divisible by 2, so that $H$ will contain a loose Hamilton cycle.
\begin{theorem} \label{HcThm}
For some constant $C$, the 2-offer 3-uniform semirandom model has a strategy that,
for $n$ divisible by $2$,
in time $t=Cn$, gives a hypergraph $H_t$ a.a.s.\ containing a loose Hamilton cycle.
\end{theorem}

This result also extends.

\begin{theorem} \label{s-uniformHc}
For any $s \ge 3$ and $r \in \set{1,2}$,
for some constant $C$, the $r$-offer $s$-uniform semirandom model has a strategy that,
for $n$ divisible by $s-1$,
in time $t=Cn$, gives a hypergraph $H_t$  a.a.s.\ containing a loose Hamilton cycle.
\end{theorem}

For both perfect matchings and Hamilton cycles, the 1-offer results follow as corollaries of the 2-offer ones but, unsurprisingly, can be obtained much more easily. 
One way to do so is touched upon in a footnote near the start of \Cref{lemmas}: 
where for both matchings and cycles our 2-offer constructions use a three-phase strategy, for the 1-offer model, a simple version of the first phase suffices.
These 1-offer constructions are presented as warm-ups near the start of \Cref{SRHmatching} (addressing matchings) and \Cref{HcSec} (Hamilton cycles).

\subsection{Outline}
\Cref{background} gives some background on the semirandom process.
\Cref{lemmas} sketches the approach we will take.
The approach relies greatly on a certain auxiliary graph, and
the section motivates, states, and proves the lemmas governing the auxiliary graph.
(Roughly: a random bipartite graph contains a large uniformly random $k$-out subgraph.
A.a.s., every large subgraph of such a graph, if it has minimum degree at least $k-1$
and parts of equal size, contains a perfect matching.)
Section~\ref{SRHmatching} proves Theorem~\ref{mainthm} by presenting a 2-offer strategy for constructing a 3-uniform hypergraph matching, and analysing it.
Subsection~\ref{sec:sunif} proves the extension to $s$-uniform hypergraphs. This is a simple modification of the argument for the 3-uniform case but requires reference to the strategy: Theorem~\ref{s-uniform} is not a black-box corollary of Theorem~\ref{mainthm}.
Section~\ref{HcSec} proves Theorems~\ref{HcThm} and~\ref{s-uniformHc}.
The strategy for constructing a loose Hamilton cycle is different in its details from that for matching,
with some new complications,
but it has the same general structure and shares many key elements.

\section{Background} \label{background}

We begin with a clarification.
Recall that our $r$-offer, $s$-uniform semirandom hypergraph is formed as follows: At each round $t$, the player is offered $r$ uniform vertices and then chooses $s-r$ additional vertices to form a hyperedge. In our proofs and previous works, there are situations where, upon seeing the offered vertices, the player decides not to build any hyperedge on them and so does not bother choosing $s-r$ vertices.  We say that the player \emph{ignores} the offer.  Formally, we can think of the player adding an arbitrary set of $s-r$ vertices and then never using the resulting hyperedge. This counts as a round and so, e.g., is one of the $Cn$ steps in Theorem~\ref{mainthm}.

\medskip

We now describe some of the main results in semirandom graphs, especially in how they relate to the present work. 

A prominent result on semirandom graphs concerns constructing a graph $G$ to contain a given spanning structure $H$, such as a perfect matching or Hamilton cycle~\cite{beneliezer2020fast}. 
It shows that for a spanner $H$ with maximum degree $\D$, this can be done in time $(3\D/2+o(\D))n$. 
Roughly speaking, the construction works by fixing an arbitrary correspondence between the vertices of $G$ and $H$, 
trying to build in $G$ the edges between each vertex $v$ and its desired neighbours $N_H(v)$, and (cleverly) patching up the vertices for which this fails. 
For the 2-offer 3-uniform semirandom model, this approach fails as soon as it fixes a correspondence and thus dictates the edges desired:
to construct a given edge $(u,v,w)$ requires an offer of some two of these vertices, and in time $O(n)$ the chance of this occurring is only $O(1/n)$.

Even for graphs, a limitation of~\cite{beneliezer2020fast} is that the $o(\D)$ term means that the coefficient of the linear running time is unknown for any fixed $\D$.
It therefore does not resolve efficient constructions for low-degree graphs of common interest, prominently including perfect matchings and Hamilton cycles.

\medskip

In the paper introducing the semirandom model, \cite{beneliezer2019semirandom}, it was shown that the semirandom process on graphs is general enough to simulate several well-studied random graph models (using suitable strategies), including the extensively studied \emph{$k$-out process}. 
Since the $k$-out model is essential to the work here, let us elaborate.

In the $k$-out process on $n$ vertices, each vertex independently connects to $k$ randomly selected vertices, resulting in a random graph with $k n$ edges (see, for example, 
\cite[Chapter~18]{Karonski_Frieze}). 
To capture the association between vertices and edges, it is often convenient to think of the process's result as a \emph{$k$-out random digraph} with $k$ random \emph{directed} edges out of each vertex.
This is easily simulated with the semirandom process by ``choosing'' vertex 1 to accompany each of the first $k$ offered vertices (constructing a directed edge from vertex 1 to each of them), choosing vertex 2 for the next $k$, etc.
(Details, such as the possibility of being offered vertex~1 when trying to construct edges out of vertex~1, are easily resolved and uninteresting.)

A linear-time semirandom construction of a \emph{perfect matching} follows from this, since a.a.s.\ the $2$-out process has a perfect matching~\cite{frieze1986maximum}.
Using the semirandom process to construct a different sort of random graph known to have a perfect matching a.a.s.\ (see \cite{pittel}), the upper bound can be improved to $(1+2/e+o(1))n < 1.73576n$~\cite{beneliezer2019semirandom}. 

In the 2-offer 3-uniform semirandom model it is equally easy to generate a \emph{random $k$-out hypergraph}, i.e., a hypergraph with $k$ random edges on (and associated with) each vertex.
Simply, to each of the first $k$ offered pairs add vertex~1,
and so on.
However, as discussed in the introduction (\Cref{sec:contribution}), it is not known whether a random $k$-out hypergraph a.a.s.\ contains a perfect matching,
and our approach to constructing a matching (and loose Hamilton cycle) is entirely different.

\medskip

For both perfect matchings and Hamilton cycles, researchers have sought the most efficient semirandom constructions possible.
For matchings, currently the most efficient construction uses time $1.20524 n$~\cite{gao2022perfect}.
On the other hand, \cite{beneliezer2019semirandom} observed that $(\ln(2)+o(1))n > 0.69314n$ is a lower bound on the number of rounds needed to create a perfect matching (indeed, to give each vertex nonzero degree), and this has since been improved to $0.93261 n$~\cite{gao2022perfect}.

Hamilton cycles in semirandom graphs have been studied in
\cite{beneliezer2019semirandom,gao2020hamilton,gao2022fully,frieze2022hamilton,ham_cycles_preprint}, and currently the best upper and lower bounds are $1.81701 n$ and $1.26575 n$~\cite{ham_cycles_preprint}.

\medskip

Another line of research on the semirandom model is in constructing a graph containing a given, fixed graph~$H$. Reminiscent of the derivation of thresholds for the presence of a fixed subgraph in a (usual) random graph,
the task is quite different from the one studied here.
For one thing, it can be accomplished in sublinear time:
where $H$ has degeneracy $d$, a copy 
can be constructed in time just larger than $n^{(d-1)/d}$ \cite{beneliezer2019semirandom},
but cannot be constructed in time just smaller than that~\cite{behague2022}.
The construction method is different from what has been done for spanning subgraphs, and what is done here:
a copy of $H$ is constructed vertex by vertex, in an order in which each vertex $v$ needs to connect to at most $d$ previous ones, using all offered vertices to try to fulfill the role of $v$.
The lower bound in \cite{behague2022} is more complicated but uses the same general ideas.
We mention this work primarily because \cite{behague2022} extends these graph results to hypergraphs, providing one of the few results for semirandom hypergraphs.
However, the extension to the 1-offer model is straightforward, while for 2 or more offers the understanding is incomplete, with the upper and lower bounds matching only for some special cases~\cite{Behague_preprint}. 

\medskip

We briefly mention a few other results.
Semirandom constructions to achieve the largest possible clique, largest possible chromatic number, and smallest possible independent set are studied in~\cite{gamarnik2023cliques}.
Sharp thresholds for the semirandom process and a more general class of processes were studied in~\cite{macrury2022sharp}.
Variants of the semirandom process have also been explored.
In~\cite{gilboa2021semi}, the first $n$ vertices offered are all distinct, as given by a random permutation; the next $n$ vertices are given by another permutation; and so on.
In~\cite{burova2022semi}, a random spanning tree of $K_n$ is presented, and the player keeps one of the edges.
In~\cite{Harjas}, $k$ random vertices are offered, and the player selects one of them, and freely chooses a second vertex.

\section{Outline, and lemmas on an auxiliary bipartite graph} \label{lemmas}

A \emph{uniformly random $k$-out bipartite multigraph}
(or just ``uniform $k$-out'' graph)
on disjoint node sets $\calA, \calB$
is a directed graph \calD in which every node in \calA has outedges to $k$ nodes in \calB
chosen uniformly at random with replacement,
and symmetrically for \calB.
The edge directions are generally ignored, here in particular, except as a convenient way to associate a vertex with ``its'' $k$ edges (namely, its $k$ out-edges). 

A uniform $k$-out auxiliary graph plays a central role in our semirandom construction.
To motivate the lemmas in this section,
we briefly explain the role each will play.
Proper explanations will come in later sections.

Throughout, we will refer to vertices of the auxiliary graph as ``nodes'',
calling vertices of the semirandom process itself ``points'', ``elementary vertices'', or occasionally ``vertices''.
We will use calligraphic letters in this section,
for general node sets \calA and \calB and a corresponding $k$-out graph \calD.
In context, later, we will define particular node sets $A$ and $B$
and a corresponding $k$-out graph~$D$.

The number of nodes in the auxiliary graph is different from the number of points in the hypergraph.
We will reserve $n$ for the latter, writing $\nn$ in the lemmas below to make it easy later to define $\nn$ in terms of $n$.
The $k$ in the lemmas is the same as that in the main argument;
indeed, we will take $k=10$ throughout.

\bigskip

Our semirandom constructions of both hypergraph matchings and loose Hamilton cycles
will begin by defining node sets \calA and \calB
(differently for the two cases).

\smallskip

\textbf{Phase~1} will
form a uniform $k$-out auxiliary graph from a subset of a set of uniformly random edges on $\calA \times \calB$.
Because the edges supplied are uniformly random,
a few nodes are bound to have too few edges,
but \Cref{k-out} shows that we can get a uniform $k$-out graph $\calD$ on \emph{most} of the nodes in \calA and \calB.
The few nodes missing from $\calD$ will correspond to points of the hypergraph we will have to deal with specially.%
\footnote{In the 1-offer model we could easily build a $k$-out graph on \emph{all} the nodes, whereupon for $k \geq 2$ the auxiliary graph a.a.s.\ has a perfect matching which, in turn, corresponds --- according to the case --- to a hypergraph perfect matching, or a hypergraph loose Hamilton cycle. 
We will note these cases as we come to them in early in \Cref{SRHmatching} and~\ref{HcSec}.
}

\smallskip

\textbf{Phase~2} will deal with these points.
The ``actions'' addressing them will have the side effect of deleting more nodes from $\calD$,
correspondingly decreasing the degrees of the remaining nodes.
We will arrange that the outdegree of a node in $\calD$
is always either $k$ or $k-1$,
by ensuring that each outdegree is reduced at most once (i.e., by one action), and by at most~1.
The ``by at most~1'' comes from
ensuring that we never delete a node with parallel inedges
(\Cref{1vx} shows that such nodes are rare),
and that we never simultaneously delete a set of nodes sharing an inneighbour
(\Cref{2vxs} shows that such sets are rare).
(Alternatively, ``by at most constant'' is immediate from the nature of the actions used,
and we could have used a larger $k$ to end with the same outdegree bound.)
The ``at most once'' is achieved by ``blocking'' a deleted node's inneighbour's outneighbours from future deletion, adding these nodes to a set $Q$ of nodes not to be deleted;
\Cref{expansion} is the basis for showing that the set $Q$ remains small.
After all the actions, \calD still contains most of \calA and \calB.

\smallskip

\textbf{Phase~3} will complete the hypergraph construction using $\calD$ as an absorber.
Walkup~\cite{walkup1980matchings} showed that
a random 2-out bipartite graph has a perfect matching a.a.s.,
and Frieze~\cite{frieze1986maximum} showed the same without the bipartiteness condition,
for 2-out graphs and multigraphs.
Here, despite the fact that \calD is neither exactly $k$-out nor uniformly random,
\Cref{Hall} shows that almost certainly Hall's condition is satisfied,
so it has a perfect matching.
The hypergraph construction will be completed by hyperedges corresponding to the matching edges.

\medskip

Throughout, we rely on McDiarmid's form of the Azuma-Hoeffding inequality~\cite{McDiarmid}.
\begin{theorem}[Azuma] \label{Azuma}
Let $X_1,\ldots X_n$ be independent random variables with $X_k$ taking values in a set $A_k$ for each $k$.
Suppose that the (measurable) function $f \colon \prod A_k \to \mathbb R$
satisfies $\abs{f(x)-f(x')} \leq c_k$ whenever the vectors $x$ and $x'$ differ only in the $k$th coordinate.
Let $Y$ be the random variable $f(X_1,\ldots,X_n)$.
Then for any $t>0$, $\Pr(\abs{Y-\E Y} \geq t) \leq 2 \exp(-2t^2/ \sum c_k^2)$.
\end{theorem}

\smallskip

\begin{lemma}[$k$-out] \label{k-out}
For all integers $k>0$ and reals $\eps>0$,
there exists a constant $C>0$ such that the following holds.
Let $\calA$ and $\calB$ be disjoint nonempty sets with $\calA \cup \calB = [\nn]$.

Let $\Psi$ be a multiset of $C\nn$ (or more) uniformly random undirected edges on
$\calA \times \calB$.
Then, a.a.s., we can delete some edges from $\Psi$, and orient the remaining ones,
to give a directed bipartite multigraph $\calD$ on parts
$\calA' \subseteq \calA$ and $\calB' \subseteq \calB$,
where
$\card{\calA'} \geq (1-\eps) \card \calA$, $\card{\calB'} \geq (1-\eps) \card \calB$,
and $\calD$ is a uniformly random $k$-out bipartite multigraph
on $\calA' \times \calB'$.
The algorithm for constructing $\calD$ is polynomial time.
\end{lemma}

The directedness of this graph is, as touched in on \Cref{background}, just a convenient way to express that each node has exactly $k$ out-edges, whose opposite endpoints are independent.

\begin{proof}
We use a Poissonization argument.
Fix $\d = \min\set{\eps/2, 0.99\,(16 e^2 \binom{k+3}{4})^{-1/2}}$.
Let $\la$ be the value such that the probability that
a Poisson random variable $Z \sim \Po(\la)$ is
less than $k+3$ is $\d/4$;
let us write this as $\Pr(\Po(\la) < k+3) = \d/4$.
Choose $C = 2.1 \la$.

Assume that $\Po(2 \la \nn) \leq C\nn$;
the probability that this fails to hold is exponentially small in~$\nn$.
Attend only to the first $\Po(2\la \nn)$ edges of $\Psi$, 
ignoring the later ones.

Randomly orient these edges.
By the splitting property of the Poisson distribution, 
the numbers of edges out of parts $\calA$ and $\calB$ are both $\Po(\la \nn)$ and they are independent.

Suppose part $\calA$ has $m$ nodes.
Further split the edges out of \calA according to which node of \calA they are directed out of. 
Since the edges are by hypothesis random, by the splitting property of the Poisson, the number of edges $Z(v)$ out of each $v \in \calA$ is distributed as $Z(v) \sim \Po(\la \nn/m)$, and the $Z(v)$ are mutually independent over all $v \in \calA$.

Choose $t$ to be the power of 2 for which $\frac12 \nn/t < m \leq \nn/t$. 
Splitting the Poisson again, write $Z(v)$ as $Z(v) = Z_1(v)+\cdots+Z_t(v)+Z'(v)$
where every $Z_i(v) \sim \Po(\la)$, $Z'(v)$ is also Poisson,
and all are independent.
The event $Z(v) < k+3$ occurs only if every $Z_i(v) < k+3$.

Let $X$ be the number of nodes of $\calA$ with degree less than $k+3$.
From the above,
$X
 =\sum_{v \in \calA} \onex(Z(v) < k+3)
 \leq \sum_{v \in \calA} \one{ (Z_1(v) < k+3) \wedge \cdots \wedge (Z_t(v)<k+3) }
 $,
so the event $X \geq \d m$ requires that,
of all $t m$ independent probability-$\d/4$ events, at least $t \d m$ hold.
Then,
$\Pr(X \geq \d m)
 \leq \Pr(\Bin(t m, \d/4) \geq t \d m)
 \leq \Pr(\Bin(\nn, \d/4) \geq \d \nn/2)
$.
With $\d$ constant,
this has probability exponentially small in~$\nn$.

Then, with exponentially small failure probability, all but a $1-\delta$ fraction of nodes in both $\calA$ and (by the same argument, and a union bound) in $\calB$
have at least $k+3$ outedges.
For any node with more than $k+3$ outedges,
randomly delete some to leave exactly $k+3$.

\medskip

Let $\calD_0$ be the resulting directed bipartite graph.
Let $X_0$ be the nodes
failing to have outdegree $k+3$.
Let $\calD$ be the $k$-core of $\calD_0 \setminus X_0$.
$\calD$ can be constructed from $\calD_0$ by initialising $X=X_0$,
successively adding to $X$ any node with
less than $k$ outneighbours in $\calD_0 \setminus X$,
or equivalently with at least $4$ outneighbours in $X$.
When there are no more such nodes,
$\calD$ is the subgraph of $\calD_0$ induced by
$\calA' = \calA \setminus X$ and $\calB' = \calB \setminus X$.

We claim that, at the end,
$\card{X \cap \calA} < 2 \d \card \calA$ and $\card{X \cap \calB} < 2 \d \card \calB$ a.a.s.
By hypothesis, $\d \leq \eps/2$, so this condition immediately implies the lemma's conclusion.
Suppose the condition does not hold.
At any time, let $S$ and $T$ be the nodes added in parts $\calA$ and, respectively, $\calB$
that bring $X_0$ to $X$.
(I.e., $S = (X \cap \calA) \setminus (X_0 \cap \calA)$.)
For the condition to be violated,
at the end we must have $\card S \geq \d \card \calA$ or $\card T \geq \d \card \calB$.
Consider the first moment that either event first occurs.
There are two symmetric cases, and we consider just the one where,
at that moment,
$\card S = \d \card \calA$ (up to integrality) while $\card T < \d \card \calB$.
Note that every node in $S$ has at least 4 outneighbours in $T \cup (X_0 \cap \calB)$.
This remains true if we expand $T$ arbitrarily (but disjoint from $X_0$)
so that $\card T = \d \card \calB$.
The probability that any such pair $S,T$ exists is at most the expected number
of such pairs, over all fixed subsets and subject to the randomness of $\calD$.
For a fixed $v \in S$, the probability $v$ has at least $4$ outneighbours in $T \cup (X_0 \cap \calB)$ is
at most $\binom{k+3}{4} \left(\ttfrac{\card {T \cup (X_0 \cap \calB)}\;}{\;\card \calB}\right)^4
  = \binom{k+3}{4} (2\d)^4$.
Then, the expected number of pairs $S,T$ is at most
\begin{align*}
\binom{\nn}{\d \nn} \binom{\nn}{\d \nn} \left(\binom{k+3}{4} (2\d)^4\right)^{\d \nn}
 & \leq \left(\frac{e \nn}{\d \nn}\right)^{2 \cdot \d \nn}
        \left(\binom{k+3}{4} (2\d)^4\right)^{\d \nn}
\\&= \left(e^2 \binom{k+3}{4} 2^4 \d^2\right)^{\d \nn}
\leq 0.99^{\d \nn}
= o(1),
\end{align*}
using the hypothesis on $\d$.

So far, we have a graph $\calD$ on parts $\calA'$ and $\calB'$ of the claimed cardinalities.
Our implementation of the core process revealed only edges directed into $X$,
so the edges in $\calD$ remain completely random.
Each node in $\calD$ has outdegree between $k$ and $k+3$.
Reveal the outdegree of each node;
if a node has outdegree greater than $k$, randomly delete outedges
so that it has degree exactly $k$.
The remaining edges are still uniformly random,
and $\calD$ is precisely $k$-out.
\end{proof}

\begin{claim} \label{1vx}
In a uniformly random $k$-out digraph with parts $\calA$ and $\calB$ both of size $\Theta(\nn)$,
the number of nodes with any double inedge is at most $\nn^{2/3}$,
with failure probability of order
$\exp(-\Omega(\nn^{1/3}))$.
\end{claim}
\begin{proof}
We consider double inedges to part~\calB; the symmetric argument applies to part~$\calA$.
Call a node in \calB ``bad'' if it has a multiple inedge.
Consider the outedges of a node $a \in \calA$ in sequence.
If an edge duplicates a previous one (and is the first to do so),
it results in a bad \calB node.
The expected number of such duplications on $a \in \calA$ is $O(1/\nn)$,
for a total expected $O(1)$ duplications over all $a \in \calA$.
Changing an outedge of any $a \in \calA$ changes the number of bad $\calB$ nodes by at most~1.
Azuma's inequality yields the claim.
\end{proof}

\begin{claim} \label{2vxs}
For any constants $k$ and $d$, 
in any $k$-out digraph with parts $\calA$ and $\calB$ both of size $\Theta(\nn)$, 
$d$ random nodes in $\calB$ have disjoint inneighbourhoods with probability $1-O(1/\nn)$, and the same holds for $d$ random nodes in~\calA.
\end{claim}
\begin{proof}
For a given $a \in \calA$, if two of the $d$ nodes in \calB have $a$ as a common inneighbour, both are in $\Np(a)$.
Since $\card{\Np(a)} = k$ while $\card \calB = \Theta(\nn)$,
this has probability $O(1/\nn^2)$.
Taking the union bound over all $a \in \calA$,
the probability that any two $d$ nodes have a common inneighbour is $O(1/\nn)$.
\end{proof}

\begin{lemma}[expansion] \label{expansion}
For any $c,k,\d>0$, let $\etad = \max\set{\frac{17}{3}\d \ln(e/\d),2\frac{1-c}{c}k\d}$.
Let $\calD$ be a  uniformly random $k$-out bipartite multigraph
 on $\calA \times \calB$,
with $\calA \cup \calB = [\nn]$
and $\card \calA , \card \calB \geq c \nn$.
Then, a.a.s.,
for every $S \subseteq [\nn]$ with $\card S \leq \d \nn$,
$$\card{\Nm_\calD(S)} \leq \sum_{v \in S} \card{\Nm_\calD (v) } \leq \etad \, \nn . $$
\end{lemma}

That is, any set of up to a $\d$ fraction of the nodes expands to at most an $\etad$ fraction.
Note:
\begin{itemize}
\item
The lemma refers to all subsets $S$ of size at most $\d \nn$ simultaneously,
not merely any given subset.
\item
We can make $\etad$ arbitrarily small by choosing $\d$ sufficiently small.
\item
The lemma implies that sets of size $\d \nn$
expand by a factor at most $\etad/\d$.
This is false for smaller sets:
there will typically be a handful of nodes of arbitrarily high indegree.
\item
The proof below uses the first-moment method.
A slightly smaller $\etad$
can be obtained by choosing (in the Poisson limit) a degree threshold above which there are at least $\d \nn$ nodes (thus encompassing the worst-case set $S$),
and considering the total of those nodes' degrees
(concentrated, and at least $\card{\Nm(S)}$).
\end{itemize}

\begin{proof}
Assume, without loss of generality, that $\card \calA \leq \card \calB$.
Let $\deg^-(v)$ denote the indegree of $v$.
For any set $S \in [\nn]$, define the volume
$\vol(S) := \sum_{v \in S} \deg^-(v)$
to be the number of edges directed into $S$.

For any set $S$,
the neighbourhood size $\card{\Nm_\calD(S)}$ is at most $\vol(S)$;
this establishes the lemma's first inequality.
For the second,
it is enough for prove the statement for
sets $S$ of size $\d \nn$.

Let $\rho = \frac{1-c}c \geq \max\set{\frac{\card \calA}{\card \calB}, \frac{\card \calB}{\card \calA}}$.
Consider a given set $S$ of size $\card S = \d \nn$, fixed with knowledge of \calA and \calB but not the random edges.
Each edge out of $B$ may be an inedge of $S \cap A$, so
\begin{align*}
    \vol(S \cap \calA) \sim 
     \Bin \left( k \card B \; , \; 
        \ttfrac{\card{S \cap \calA} }{ \card \calA} \, \right) .
\end{align*}
The symmetric statement for $\vol(S \cap \calB)$ and the two volumes are independent, so 
$\vol(S)$ is a sum of independent Bernoulli random variables with expectation
\begin{align*}
\E(\vol(S)) &= 
  k \card{S \cap \calA} \, \tfrac{\card \calB}{\card \calA} + 
  k \card{S \cap \calB} \, \tfrac{\card \calA}{\card \calB}
  \;\leq k \; \card S \max\set{\tfrac{\card \calA}{\card \calB},\tfrac{\card \calB}{\card \calA}}
  \; \leq \; \rho k \d N .
 \end{align*}
For convenience we may artificially add more independent Bernoulli variables 
to obtain a random variable $Z$ stochastically dominating $\vol(S)$ and with $\E Z = \rho k \d \nn$.

Let $\eta = \etad$.
Apply the Chernoff-type inequality
$
\Pr(Z \ge \E Z + t) \le \exp \left( - \frac {t^2}{2(\E Z + t/3)} \right)
$
\cite[eq.~(2.5) and Theorem~2.8]{JLR}
to $\Pr(Z \geq \eta \nn)$.
This means taking $t = \eta \nn - \E Z = (r-1) \rho k \d \nn$,
where $r=\eta /(\rho k \d ) \geq 2$ by definition of $\etad$.
Then,
\begin{align*}
\Pr(Z \ge \eta \nn)
 &\leq
 \exp \left( - \frac {(r-1)^2(\E Z)^2} {2(1+\frac{r-1}{3}) \E Z} \right)
 \; \leq \;
 \exp \left( - \frac {(r/2)^2} {2(\frac{2r}{3})} \E Z \right)
 \; = \;
 \exp \Big( -\tfrac{3}{16} r \cdot \rho \d k \nn \Big)
 .
\end{align*}
The number of choices for the set $S$ is
\begin{align*}
\binom \nn {\d \nn}
  \leq \left( \frac{e \nn}{\d \nn} \right)^{\d \nn}
  = \exp( \d \nn \ln(e/\d)) .
\end{align*}
By the first-moment method, the probability that any set $S$ violates the claimed condition is at most
\begin{align*}
\binom \nn {\d \nn} \Pr(Z \ge \eta \nn)
 &\leq
 \exp \Big( \d \nn \cdot \Big( \ln(e/\d) - \frac{3}{16} r \rho k \Big) \Big)
 \\&=
 \exp \Big( \d \nn \cdot \Big( \ln(e/\d) - \frac{3}{16} \frac{\eta}{\d} \Big) \Big)
 = \exp(-\Omega(\nn))
 \; = \;
 o(1) ,
\end{align*}
the penultimate equality by definition of $\etad$.
\end{proof}

\begin{lemma}[Hall] \label{Hall}
Given $k \geq 10$ and
$\rx < 1- 2/(1+3/e)$. (Having $\rx \leq 0.049$ suffices.)
Let $\calD$ be a uniformly random $k$-out bipartite digraph on parts $\calA$ and $\calB$.
Consider the subgraph $\calD'$ induced by a pair of subsets
$\calA' \subseteq \calA$, $\calB' \subseteq \calB$,
where
\begin{enumerate}[(i)]
\item $\card{\calA'} + \card{\calB'} \geq (1-\rx) (\card \calA + \card \calB)$, \label{Hall1}
\item $\card{\calA'} = \card{\calB'}$; and \label{Hall2}
\item every node in $\calD'$ has outdegree at least $k-1$. \label{Hall3}
\end{enumerate}
\noindent
Then, a.a.s.\ in $\card \calA$
(with probability tending to 1 as $\card\calA \to \infty$)
every such induced subgraph $\calD'$
contains a perfect matching.
\end{lemma}

Note that here ``a.a.s.\ in $\card \calA$'' is equivalent to a.a.s.\ in 
$\card \calB$, $\card{\calA'}$, or $\card{\calB'}$.
Also, the hypothesis $k \geq 10$ can be weakened to $k \geq 7$, at the expense of a slightly more complicated proof, and likely further.

\begin{proof}
Let $n := \card{\calA'} = \card{\calB'}$.
If $\calD'$ does not have a perfect matching (ignoring the edge directions)
then there are Hall sets $S \subseteq \calA'$ and $T \subseteq \calB'$
such that $N(S) \subseteq T$ and $\card T = \card S-1$.
Write $\bar S$ for $\calA' \setminus S$ and $\bar T$ for $\calB' \setminus T$.
If $S,T$ is a Hall pair, so is $\bar T, \bar S$.
Since $\card S + \card{\bar T}= n+1$,
at least one of them must be at most $(n+1)/2$.
By symmetry, it suffices to consider a Hall pair $S,T$
where $S \subseteq \calA$ has cardinality $s \leq (n+1)/2$.

Henceforth we ignore $\calA'$ and $\calB'$.
We are simply interested in the existence of a Hall pair $S,T$
with $S \subseteq \calA$, $T \subseteq \calB$,
$\card S = s \leq (n+1)/2$, and $\card T = s-1$.
The probability that $\calD$ allows such a pair
is at most the expected number of such pairs,
which is at most the number of pairs times the probability
that a fixed pair has the Hall property.

Since $2n = \card {\calA'} + \card {\calB'} \geq (1-\rx) (\card{\calA}+\card{\calB})$,
we have that
$\card{\calA}+\card{\calB} \leq \frac{2}{1-\rx} n$
and thus
\begin{align*}
\card \calA
 & \leq  \tfrac{2}{1-\rx} n - \card \calB
 \; \leq \; \tfrac{2}{1-\rx} n - \card{\calB'}
 \; = \; \parens{ \tfrac{2}{1-\rx} -1} n
 \; \leq \; \tfrac{3}{e} n ;
\end{align*}
the last inequality follows from the hypothesis on $\rho$ and is used in the next calculation.
The same holds for $\calB$.
Also, $\card \calA$ and $\card \calB$ are both $\Theta(n)$,
and henceforth we work asymptotically in $n$.

For each $v \in S$, all but one of $v$'s outedges must be directed into $T$.
The probability that there is such a Hall pair $S,T$, then,
is at most the number of choices of
$S \subseteq \calA$ with $\card S = s$,
of $T \subseteq \calB$ with $\card T = s-1$,
and (by hypothesis~\ref{Hall3}) of one edge per element of $S$,
times the probability that such a combination has
all other edges of $S$ directed into $T$.
The probability there is any such pair is at most
\begin{align}
\binom{k}{1}^s \binom{(3/e)n}{s} \binom{(3/e)n}{s-1} \left(\frac{s-1}n\right)^{(k-1)s}
&\leq
k^s \parens{\frac {3n} s}^{s} \parens{\frac {3n} {s-1}}^{s-1} \parens{\frac{s-1}n}^{(k-1) s}  \notag
\\ &\leq
(9k)^s  \parens{\frac{s-1}n}^{(k-3) s+1} \notag
\\ &\leq
\frac s n \parens{ 9k \parens{\frac{s-1}n}^{k-3}}^s \label{smalls}
.
\end{align}

The sum of \eqref{smalls} over $s$ from $2$ to $10 \ln n$,
even ignoring the power $s$, is
$O(\ln n) \cdot O\parens{\parens{\frac{\ln n}n}^{k-2}} = o(1)$.
For larger $s$, up to
$s=(n+1)/2$, we have $(s-1)/n < 1/2$,
so, since $k \geq 10$,
$
9k(\tfrac{s-1}n)^{k-3}
  \; \leq \; 9k (\tfrac12)^{k-3}
  \; < \; 0.8
  \;\ < \; \exp(-1/5)
$.
Each term in \eqref{smalls} is thus of order $O(e^{-s/5}) = O(e^{-2 \ln n})=O(n^{-2})$,
for a total of $O(1/n)=o(1)$.
\end{proof}

\section{Matching Strategy and Analysis} \label{SRHmatching}

We construct a hypergraph $H$, initially empty and eventually containing a perfect matching $\calM$, on vertex set $V=[n]$;
we assume that $n$ is divisible by 3.
Partition $V$ into a set $A_0$ of vertices called ``apexes'', and a set $B_0$ of pairs of vertices, called ``base pairs'' or just ``bases'', with
\begin{align}
  \card {A_0} &= \frac{n}{3}+4 \gapd n &
  \card {B_0} &= \frac{n}{3} - 2 \gapd n , \label{cardAB}
\end{align}
where $\gapd$ is a constant chosen small enough to satisfy various claims below.
The two members of a base pair are sometimes called ``partners''.
Integrality is not an issue; all that needs to be exact is that
$\card{A_0}+2\card{B_0}=n$.

\medskip

As a warmup exercise, consider the 1-offer model, and take the set sizes to be $\card{A_0}=\card{B_0}=n/3$. 
Construct a random 2-out bipartite graph as follows. 
Focussing on the first point $a_1 \in A_0$, on offer of any $b \in \bb \in B_0$, choose as third point the partner $b'$ of $b$ in $\bb$, construct the hyperedge $\set{a_1} \cup \bb$, and correspondingly add the edge $\set{a_1,\bb}$ to the auxiliary graph. 
Repeat this to build a second random auxiliary graph edge on $a_1$, and the corresponding hyperedge; then repeat to build two edges on $a_2$, on $a_3$, and so forth.
Similarly, focussing on the first point $\bb_1 \in B_0$, on offer of any $a \in A_0$, build the hyperedge $\set{a,\bb_1}$ and the corresponding auxiliary graph edge $(a,\bb_1)$, repeating to make two such random edges (and corresponding hyperedges) on each node in $B_0$.
This can be completed in linear time.
(It will take about $2n$ trials to be offered $\frac23 n$ nodes in $A_0$ to make the requisite edges to $B_0$, and about $n$ trials to be offered $\frac23 n$ points in nodes of $B_0$ to make the edges to $A_0$, and these can even be done in parallel.)
At this point the auxiliary graph is a random 2-out graph and thus a.a.s.\ has a perfect matching. 
By construction, the corresponding hyperedges form a hypergraph perfect matching, as desired. 

In the 2-offer model we will be unable to make \emph{every} node of $A_0$ and $B_0$ $k$-out, 
and will thus take the 3-phase approach outlined in \Cref{lemmas}.

\subsection{Phase 1: A robust matching structure}
In this phase we construct the digraph $D_1$ introduced above,
and in tandem the corresponding hypergraph $H$.

For a sequence of
semirandom offers, do the following.
For each offer of the form $\set{a,b}$,
where $a \in A_0$ and $b \in \bb \in B_0$ is a member of a base pair,
choose as third point the partner of $b$.
This defines a hyperedge $\set{a} \cup \bb$, which we add to our hypergraph $H$,
and an (undirected) auxiliary graph edge
$\set{a,\bb}$, which we add to a set $\Psi$.
Semi-random offers not of the specified form are ignored.
(Recall what this means from the beginning of Section~\ref{background}.)

Note that each $(a,\bb)\in \Psi$ is uniformly random in $A_0 \times B_0$, with replacement.
Given $\gapd$ and $k$, let $C$ be the corresponding value in \Cref{k-out}.
Each semirandom offer is of the desired form w.p.\ (with probability)
$\frac13 \cdot \frac23+\frac23 \cdot \frac13 +O(\gapd) = 4/9+O(\gapd) > 1/3$
for $\gapd$ sufficiently small,
so in $3Cn$ semirandom offers,
a.a.s.\ there are at least $Cn$ of the stipulated form,
exceeding the lemma's $CN$ (since by \eqref{cardAB}, $N$ is about $\frac23 n$).

Should this a.a.s.\ statement fail,
we consider the entire construction to have failed.
We will take the same approach throughout,
proceeding on the assumption that all the a.a.s.\ statements hold.

It follows from \Cref{k-out} that from this set $\Psi$ of undirected edges
we can a.a.s.\ construct a bipartite digraph $D_1$
on parts $A_1 \subseteq A_0$ and $B_1 \subseteq B_0$,
with $\card{A_1} \geq (1-\gapd) \card{A_0}$ and
$\card{B_1} \geq (1-\gapd) \card{B_0}$,
and $D_1$ is uniformly $k$-out on these node sets.

By construction, $H$ contains the hyperedges corresponding to the edges of $D_1$.
We will discard (or ignore) all other hyperedges in $H$
so that there is a one-to-one correspondence between
the edges in $D_1$ and the hyperedges in $H$:
they are two representations of the same thing.

Let $X_1 = (A_0 \bk A_1) \cup (B_0 \bk B_1)$ be the set of
``failed'' nodes missing from $D_1$.
By Lemma~\ref{k-out} and~(\ref{cardAB}), we have
\begin{align}
\card{A_1} &\geq (1-\gapd)\card{A_0} > \frac{n}{3}+3\gapd n    \label{cardA1} \\
\card{B_1} &\geq (1-\gapd)\card{B_0} > \frac{n}{3}-3\gapd n    \label{cardB1} \\
\card{X_1} &\leq \gapd(\card{A_0}+\card{B_0}) < \gapd n    \label{cardX1} .
\end{align}

Let $\Xflat_1$ be the set of all the points (hypergraph vertices) contained in $X_1$, so $\Xflat_1 = (A_0 \cap X_1) \cup \bigcup_{\bb \in (B_0 \cap X_1)} \bb$.
Letting $\ax=\card{A_0 \cap X_1}$ and $\bx  = \card{B_0 \cap X_1}$,
note that
\begin{align}
\tcard{\Xflat_1} &= \ax+2\bx .     \label{cardXflat}
\end{align}

\subsection{Phase 2: Matching exceptional points}
This phase will place every point of $\Xflat_1$ into a hyperedge to be included in the matching~$\calM$,
while keeping $D$ robust enough that it will contain a matching.

At the beginning of this phase, $A=A_1$, $B=B_1$, $D=D_1$, $X=X_1$, and $\Xflat = \Xflat_1$.
Along the way, $A,B,D,X, \Xflat$ will change;
e.g., points will be removed from $\Xflat$ as they are placed in hyperedges of $\calM$.

We first consider the operations we will perform in this phase,
deferring the question of how many semirandom offers are required.

\subsubsection{Phase 2a}

Here we treat all the points in $X_1$, from apexes and base pairs alike.
We will be more precise in a moment,
but roughly speaking, any time we get a semirandom offer consisting of two points $a,a'\in A$, we choose an arbitrary $x \in \Xflat$, defining a hyperedge $e=\set{a,a',x}$,
and add $e$ irrevocably to \calM.
The failed vertex $x$ is now in a hyperedge in \calM,
but $a$ and $a'$ must be deleted from $A$ and $D$:
neither is any longer available to be combined with a base pair.

\medskip

As said in the Outline (\Cref{lemmas}),
we will ensure that, as we delete nodes from $D$,
every node in $D$ retains outdegree at least $k-1$.
To do this, we ``block'' some nodes of $D$ from future deletion;
we let $Q$ denote the set of blocked nodes.
To avoid that any single node's deletion decreases an inneighbour's outdegree by 2 or more,
we intialise $Q$ to be the set of nodes with parallel inedges.
To avoid that the outdegree of any node $u$ is decreased more than once,
if ever an outneighbour $v$ of $u$ is deleted,
we block all other outneighbours of $u$ from future deletion.
That is, whenever a node $v$ is deleted from~$D$,
we add $\Np(\Nm(v))$ to~$Q$.
All that remains is to ensure that in any single action,
we do not delete two nodes with a common inneighbour.

\medskip

With this out of the way, we can be precise.
For any semirandom offer consisting of two points $a,a'\in A_0$,
we choose an arbitrary $x \in \Xflat$, defining a hyperedge $e=\set{a,a',x}$.
We discard $e$ unless the following conditions both hold:
\begin{enumerate}[label={(M2a-C\arabic*)},ref={(M2a-C\arabic*)},leftmargin=3cm]
\item $a,a' \in A \setminus Q$. \label{m2ac1}
\item $\Nm(a)$ and $\Nm(a')$ are disjoint. \label{m2ac2}
\end{enumerate}
(Read the label as M for matching, 2a for the phase, and C for ``condition'';
in the next part, A is for ``action''.)

If the conditions hold, we take the following action:
\begin{enumerate}[label={(M2a-A\arabic*)},ref={(M2a-A\arabic*)},leftmargin=3cm]
\item Add $e$, irrevocably, to \calM.
\item Delete $x$ from $\Xflat$ and $a$ and $a'$ from $A$ (and hence also from $D$).
\item Add $\Np(\Nm(\set{a,a'}))$ to $Q$,
blocking the outneighbours of the inneighbours of the deleted nodes.
\end{enumerate}

Each such action resolves one point in $\Xflat$,
so by \eqref{cardXflat}, $\ax+2\bx$ actions are required.
Each action deletes 2 nodes of part $A$ and 0 nodes of part $B$, so
at the end of this phase,
$\card B = b_0 - \bx$ and
$\card A
 = (a_0-\ax) - 2(\ax+2\bx)
 = a_0 - 3 \ax - 4\bx$.
Also,
\begin{align}\label{a-b}
\card{A}-\card{B}
 &= a_0-b_0-3\ax-3\bx
 \; = \; a_0-b_0-3\card{X_1}
 \\& \geq a_0-b_0-3\gapd n
 \; = 3 \gapd n > 0
 \qquad\text{(by \eqref{cardX1} and  \eqref{cardAB})} .
\end{align}
Thus,
$A$ is larger than $B$ at the end of Phase~2a.

\subsubsection{Phase 2b}
This phase makes the sizes of $A$ and $B$ exactly equal so that they can be matched in $D$.
Recall that at the end of Phase~2a, $\card A > \card B$.
To even them up, we will generate hyperedges each consisting of three vertices from $A$,
deleting the points from $A$ and adding the hyperedge to~\calM.

To be precise,
each time we get a semirandom offer consisting of two vertices $a,a'\in A_0$,
we define a hyperedge $e=\set{a,a',a''}$ for a uniformly random
$a'' \in A \setminus \set{a, a'}$.
We discard the offer unless the following conditions both hold:
\begin{enumerate}[label={(M2b-C\arabic*)},leftmargin=3cm]
\item $a,a',a'' \in A \setminus Q$. \label{m2bc1}
\item $\Nm(a)$, $\Nm(a')$, and $\Nm(a'')$ are disjoint. \label{m2bc2}
\end{enumerate}

If the conditions hold, we take the following action:
\begin{enumerate}[label={(M2b-A\arabic*)},leftmargin=3cm]
\item Add $e$, irrevocably, to \calM.
\item Delete $a, a'$, and $a''$ from $A$ (and hence also from $D$).
\item Add $\Np(\Nm(\set{a,a',a''}))$ to $Q$,
blocking the outneighbours of the inneighbours of the deleted nodes.
\end{enumerate}

From \eqref{a-b}, at the beginning of Phase~2b we have
$\card{A}-\card{B} \equiv a_0-b_0 \equiv a_0+2 b_0 = n \equiv 0 \pmod 3$.
Since each action deletes 3 vertices from $A$,
we can make $\card A = \card B$ exactly.
The precise number of actions needed is not critical
--- it is obviously
$O(\gapd n)$ ---
but from \eqref{a-b} and \eqref{cardAB} it is
\begin{align}\label{nr1b}
\frac{\card{A}-\card{B}}3
 = \frac{a_0-b_0}3-\card{X_1}
 \leq \frac{a_0-b_0}3
 = 2 \gapd n.
\end{align}

\subsubsection{Analysis and parameter choices} \label{analysis}

How many nodes can ever be blocked?
Initially, $Q=Q_1$ consists of the nodes with double inedges.
By \Cref{1vx}, $\tcard{Q_1} = O(n^{2/3})$,
with failure probability exponentially small in $n$.
As usual, we will declare failure of the construction if this fails.

Recall that when an action deletes a node $v$ from $D$, it adds $\Np(\Nm(v))$ to~$Q$.
Phase~2a takes $\ax+2\bx$ actions, each deleting 2 nodes from~$D$;
note that $\ax+2\bx \leq 2 \tcard{X_1} \leq 2\gapd n$, by \eqref{cardX1}.
Phase~2b takes at most $2 \gapd n$ actions, each deleting 3 nodes from~$D$.
Altogether, they delete at most $10 \gapd n$ nodes from~$D$.

The number of inneighbours of these nodes is bounded by \Cref{expansion},
which will be our focus for the next few paragraphs.
The lemma's $N$ is
$\card{A_1}+\card{B_1} > \frac23 n$ by \eqref{cardA1} and \eqref{cardB1},
so $10 \gapd n$ represents at most a $15\gapd$ fraction of $N$.
Thus we take the lemma's $\delta$ --- call it $\dd$ --- equal to $15\gapd$.

The two parts $A_1$ and $B_1$ are of nearly equal size,
so we may take \Cref{expansion}'s $c$ as $\frac13$.
We will make many assertions of this sort, so let us be more precise this one time.
In this case, from \eqref{cardAB}, \eqref{cardA1}, and \eqref{cardB1} we have
$(\frac13+3\gapd) n \leq \card{A_1} \leq \card{A_0} = (\frac13+4\gapd) n$ and
$(\frac13-3\gapd) n \leq \card{B_1} \leq \card{B_0} = (\frac13-2\gapd) n$.
We will choose $\gapd$ small enough that it is obvious that
each of $A_1$ and $B_1$ has size at least $1/3$rd that of their union.
More to the point, forgetting the precise multiples of $\gapd$ here,
we will have expressions like $\card{A_1} = (\frac13+O(\gapd))n$,
and it will be clear that $\gapd$ \emph{can} be chosen sufficiently small
to satisfy the various assertions.

Continuing, by \Cref{expansion},
the number of inneighbours is at most
$\eta_{1/3,k}(\dd) N$.
Each inneighbour has $k$ outneighbours, so the number of outneighbours
is at most $k \: \eta_{1/3,k}(\dd) N$.
Since $\card{Q_1} = o(n) = o(N)$, we have that at all times
$\card Q \leq k \: \eta_{1/3,k}(\delta) N + o(N)$.
We will contrive that $k \: \eta_{1/3,k}(\dd) \leq \tfrac1{10}$,
so that $\card Q \leq \tfrac N 9$
and $\card Q/ \card{A_1} \leq (\tfrac N 9)/(\tfrac N 3) \leq \tfrac13$
and symmetrically $\card Q/ \card{B_1} \leq \tfrac13$.

This dictates our parameter choices.
We take $k=10$ throughout, for purposes of \Cref{Hall}.
Then, in \Cref{expansion}, choosing $\delta =\dd = \num{0.000 15}$ suffices,
giving $k \: \eta_{1/3,k}(\dd) \leq \tfrac1{10}$ as desired.
(The value of $\eta$, defined in Lemma~\ref{expansion}, is determined by the first term in the max;
the second term, $40 \: \delta$, is smaller.)
Since we set $\dd=15 \gapd$, taking $\gapd = \num{0.000 01}$ suffices.

Recall that Phase~1 used $3Cn$ semirandom offers, with $C$ given by \Cref{k-out}.
Now that we have fixed $\gapd=\num{0.000 01}$,
the proof of \Cref{k-out}
fixes its $\delta=\gapd/2$ (the second term in the min is much larger).
In turn, to give $\Pr(\Po(\la) < k+3) = \delta/4 = \gapd/8 =\num{0.000 001 25}$,
it fixes $\la$ to be between 37 and 38, for $C=2.1 \la < 80$.
Thus, $3Cn < 240n$ semirandom offers suffice for Phase~1.

\smallskip

Looking ahead, we will invoke \Cref{Hall} to prove existence of a perfect matching in~$D$.
The lemma's \calA and \calB are the parts  $A_1$ and $B_1$ at the start of Phase~2,
its $\calA'$ and $\calB'$ are the parts $A$ and $B$ at the end of Phase~2,
and its $\rho$ is (an upper bound on) the fraction of nodes deleted,
so we take $\rho=\delta$.
With the choices above, then, $\rho \leq 0.049$, satisfying the hypothesis.

\bigskip

Having controlled the size of $Q$, we are ready to consider the number of semirandom offers needed in Phase~2.

In Phase~2a, that the first point offered is in $A_0$ occurs with probability just above $1/3$.
The cardinality of $A$ is nearly that of $A_0$
and we have arranged that $\card Q \leq \frac1{3} \card{A_1}$, so
\begin{align}\label{m2ac1-x}
 \Pr(a \in A \bk Q)
   & \geq \tfrac13 \cdot (1-O(\gapd)- \tfrac1{3})
   \; > \; \tfrac1{5} .
\end{align}

Independently (almost), the same holds for $a'$,
so \ref{m2ac1} holds with probability at least $1/25$.%
\footnote{%
We would have independence if the offered $a'$ were independent of $a$,
allowing the possibility that $a'=a$ (in which case we would reject the offer).
We can simulate this independence by tossing a coin and with probability $1/n$
replacing the offered $a'$ with $a$.
In this simulated model, \ref{m2ac1} and distinctness hold with probability
at least $\Pr(a \in A \bk Q)^2 - 1/n$,
here simply $1/25$ by absorbing the $1/n$ into slack in \eqref{m2ac1-x}.
Since the simulated substitution of $a'$ for $a$ is never advantageous,
\ref{m2ac1} holds with probability at least $1/25$ in the original model.%
}
By \Cref{2vxs}, the chance that \ref{m2ac2} fails to hold is $O(1/n)$.
Thus, each offer satisfies the conditions with probability at least $1/26$.
Since at most $2\gapd n$ actions are needed (see \eqref{cardX1}),
a.a.s., $27 \cdot 2 \gapd n$ offers suffice.

Similarly, in Phase~2b, an offer plus the added point $a''$
satisfies condition \ref{m2bc1} with probability more than $(1/5)^3$
(even if we chose $a''$ in [$n$] rather than $A$),
and again \ref{m2bc2} fails with probability $O(1/n)$,
so each offer results in action with probability at least $1/126$.
By \eqref{nr1b}, the number of actions needed is at most $2\gapd n$,
so, a.a.s., $127 \cdot 2 \gapd n$ offers suffice.

It is clear that, with control of $\card Q$,
each offer in Phases~2a and 2b results in action with probability $\Theta(1)$,
and thus $O(n)$ offers suffice.
Numerically, Phases~2a and 2b together require at most
$(27+127) \cdot 2 \gapd < \num{0.01}n$ offers,
so $241 n$ offers suffice for Phases~1 and~2 (thus for the whole algorithm).
We have made no effort to optimise the constants.

\bigskip

We use $A_2, B_2, D_2$ to denote $A,B,D$ at the end of Phase 2.
At this point, all points outside of $A_2, B_2$ appear in hyperedges of $\calM$.
Since no nodes were deleted from $B$ during Phase~2, and applying~(\ref{cardB1}), we  have:
\begin{align}\label{eP2}
\card{A_2} &= \card{B_2} = \card{B_1} > \frac{n}{3}-3\gapd n.
\end{align}

\subsection{Phase 3: Matching the bulk of the points}\label{sHall}

We complete the hypergraph matching $\calM$ with hyperedges
corresponding to the edges of a perfect matching in $D_2$;
\Cref{Hall} proves that a.a.s.\ $D_2$ has a perfect matching.

We apply the lemma with $\calD=D_1$ and $\calD'=D_2$,
and accordingly $\calA=A_1$, $\calA'=A_2$, and $\calB=\calB'=B_1=B_2$.

We already noted in \Cref{analysis}
(and confirmed in \eqref{eP2})
that hypothesis \ref{Hall1} is satisfied:
the fractional size difference between $A_1$ and $A_2$ is $O(\gapd)$ and far less than $1/100$.
Hypothesis \ref{Hall2}, that the two parts are of equal cardinality,
is satisfied by construction; see just before \eqref{nr1b}.
Hypothesis \ref{Hall3}, that every node degree is $k$ or $k-1$,
is also satisfied by construction,
specifically the various action conditions based on the blocked set~$Q$.
Thus, a.a.s.\ $D_2$ contains a perfect matching.

This concludes the proof of Theorem~\ref{mainthm}.

\subsection{Extension to \texorpdfstring{$s$}{s}-uniform hypergraphs} \label{sec:sunif}
We now establish Theorem~\ref{s-uniform}.
As already explained, constructing a perfect matching in
the 1-offer $s$-uniform model
can trivially be simulated in the 2-offer $s$-uniform model,
simulating a 2nd ``offer'' vertex with the extra ``chosen'' vertex.

Constructing a matching in the 2-offer $s$-uniform model
cannot be done using
the 2-offer 3-uniform model
as a black box, but can be done in exactly the same way.
In short, in Phase~0, take a bit more than $n/s$ vertices as apexes $A_0$, arbitrarily partitioning the rest into a bit fewer than $n/s$ ``base'' groups each of size $s-1$, comprising $B_0$.

\subsubsection{Phase 1}
In Phase 1, when offered an apex $a$ and a point $b$ from a base $\bb$,
propose the hyperedge $a \cup \bb$
and correspondingly an auxiliary-graph edge $\set{a,\bb}$,
directed one way or the other.
As before, with such proposals we can construct
a directed auxiliary graph $D = D_1$,
bipartite with parts $A_1$ and $B_1$ comprising most of $A_0$ and $B_0$ respectively,
with $D_1$ a uniformly random $k$-out digraph on $A_1$ and $B_1$.
Let $X=X_1 = (A_0 \cup B_0) \bk (A_1 \cup B_1)$
be the nodes missing from $D$,
and $\Xflat$ the set of points in $X$.

\subsubsection{Phase 2}
In Phase~2a, until $\Xflat$ is empty, when offered two points in $A_0$,
propose a hyperedge consisting of them and any $s-2$ points in $\Xflat$.
If in a final step $0< \tcard{\Xflat} < s-2$,
supplement the points of $\Xflat$ with distinct unblocked points from $A$.
Correspondingly, propose the action of deleting the two apexes from
part $A$ of the auxiliary graph $D$.

As we did for matchings, we block $\Np(\Nm(v))$ for any node $v$
deleted from $D$.
We accept a proposed action
(and accept the corresponding hyperedge for $\calM$)
if the offered vertices lie in the current set $A$ and are not blocked.
By analogy with the argument around \eqref{m2ac1-x},
by controlling the size of $\Xflat_1$
we can ensure that the number of nodes deleted and blocked is small,
ensuring in turn that each proposed action is accepted w.p.\
at least
$(1/s\cdot 0.98)^{s-2} = \Omega(1)$.
Also as for matchings, we choose the sizes $\card{A_0}$ and $\card{B_0}$
so that, a.a.s., when Phase~2a ends, $\card A \geq \card B$.

In Phase~2b, when offered 2 vertices $a,a' \in A_0$,
propose a hyperedge consisting of them and $s-2$
distinct points in $A \setminus \set{a,a'}$ and
correspondingly propose the action of deleting these $s$ points from
part $A$ of $D$.
Accept the proposal if the two offered points were in $A \bk Q$;
in this case, delete these $s$ apexes from $D$,
and accept the hyperedge for $\calM$.

Again, each proposed action is accepted w.p.\ at least
$(1/s\cdot 0.98)^{s-2} = \Omega(1)$.
Stop when the number of apexes is equal to the number of bases,
$\card A = \card B$.

For the stopping condition,
since at the end of Phase~2a, $\card A \geq \card B$,
and Phase~2b decreases the size of $A$ without changing $B$,
it is clear that at some time we get $\card A \leq \card B$.
At this time,
we get equality if $\card A \equiv \card B \pmod s$,
and this is so.
Let $V=[n]$ be the set of all points, with $\card V \equiv 0 \pmod s$.
Let $\flatVD$ be the set of points in $D$,
i.e., the points in $A$ and the points belonging to bases in~$B$.
At this step, the points of $V \setminus \flatVD$ are all in hyperedges in $\calM$, so
$\card{V \bk \flatVD} \equiv 0 \pmod s$.
Thus, at this step,
\begin{align}\label{match-s}
\card A - \card B
 &\equiv \card A + (s-1) \card B
 = \card{\flatVD}
 = \card V - \card{V \bk \flatVD}
 \equiv 0
 \pmod s .
\end{align}

\subsubsection{Phase 3}
Finally, construct a matching in $D$ (ignoring the edge directions).
The corresponding hypergraph edges complete a perfect matching in the hypergraph.

As for $s=3$,
it is enough to ensure that when constructing $D_1$,
$\card X \leq \gapd n$ for some sufficiently small constant $\gapd$.
The number of Phase~2 steps is $O(\gapd n)$
(the $O(\cdot)$ expression hiding some constant independent of $\gapd$),
so the number of nodes deleted from $D$ is also $O(\gapd n)$.
By \Cref{expansion},
$Q$ remains of size at most $k \: \etad(O(\gapd)) n$ (plus $O(n^{2/3})$),
and this can be made an arbitrarily small fraction of $n$ by choosing
$\gapd$ sufficiently small.
If follows that each action is accepted w.p.\ $\Omega(1)$.

\section{Loose Hamilton cycles} \label{HcSec}

Now we turn to the proof of Theorem~\ref{HcThm}.
We will demonstrate a strategy that,
within a linear number of steps,
a.a.s.\ produces a 3-uniform hypergraph with a loose Hamilton cycle.
The vertex set is $[n]$ for some even $n$.
(The number of vertices must be even if a loose Hamilton cycle exists.)
The odd points will occur twice each in the loose cycle,
the even points just once, in the form
\begin{align} \label{idealHc}
(1,a_1,3), \; (3,a_2,5), \; \ldots, \; (n-1,a_{n/2},1) ,
\end{align}
where we write each edge as an ordered triple to indicate the roles of the three points,
and $a_1,a_2,\ldots,a_{n/2}$ is a permutation of $2,4,\ldots,n$.
We will not be able to get exactly such a cycle,
with the odd numbers appearing in order,
but will obtain a Hamilton cycle that differs from it on very few edges.

The approach is similar to that for generating a perfect matching in \Cref{SRHmatching}.
We define apex and base sets,
each of cardinality exactly $n/2$:
\begin{align}
A_0 &= \set{2,4,\ldots, n}
&
B_0 &= \set{\set{1,3}, \set{3,5},\ldots \set{n-1,1}} .
\end{align}
Here the bases are overlapping,
in order to build a cycle like that of \eqref{idealHc},
where for hypergraph matchings they were disjoint.

The next steps are structurally the same as for matchings in \Cref{SRHmatching}.
As in \Cref{SRHmatching}, in the 1-offer model we could (exactly as in \Cref{SRHmatching}) make a 2-out graph on $A_0$ and $B_0$; it would a.a.s.\ contain a perfect matching; and that auxiliary graph matching corresponds to a loose Hamilton cycle in the hypergraph. 
In the 2-offer model it is not so easy.

We will construct a directed bipartite graph on $A_0$ and $B_0$,
but it will not have a perfect matching because a small number of nodes fail to have sufficiently high degree (and cause other nodes to ``fail'' as well).
We deal with the failed nodes first, and in doing so add certain hyperedges, irrevocably, to our cycle-in-the-making $\calC$.
($\calC$ plays the role that $\calM$ did for hypergraph matchings.)
When the failed nodes are dealt with, most of our bipartite graph will remain, and what remains will have a perfect matching.
The matching edges correspond to hyperedges in $\calH$, and they complete $\calC$ to a loose Hamilton cycle.

\subsection{Phase 1}
In the first phase, we build a graph $D$ in the same manner as we did when building a perfect matching in Section~\ref{SRHmatching}.

For each offer of the form $\set{a,b}$,
where $a \in A$ and $b \in \bb \in B$ is a member of a base pair,
choose as third point the partner of $b$
and add $\set{a,\bb}$ to a set $\Psi$. (We ignore any offers that are not of this form.)

$\Psi$ forms a set of edges of a bipartite graph on $A_0\times B_0$. Note that a perfect matching on this graph would yield a loose Hamilton cycle of the form in~\eqref{idealHc}.

For any $\gapd>0$,
Lemma~\ref{k-out} ensures,
exactly as in Phase 1 of Section~\ref{SRHmatching},
that within $O(n)$ steps we can
construct a directed bipartite graph $D=D_1$
with parts
$A_1 \subseteq A_0$ and $B_1 \subseteq B_0$,
such that $D_1$ is a uniformly random $k$-out bipartite multigraph
on $A_1 \times B_1$,
and a.a.s.\
$\card{A_1} \geq (1-\gapd) \card {A_0}$
and $\card{B_1} \geq (1-\gapd) \card {B_0}$.
We refer to the apexes in $A_1$ as \emph{good}
and those in $\AX = A_0 \setminus A_1$ as \emph{failed},
and likewise for bases.
We have
\begin{align}  \label{Hcardax}
\ax &= \card\AX \leq \gapd n/2
&
\bx &= \card\BX \leq \gapd n/2 .
\end{align}

\subsection{Phase 2}

We introduce an additional structure, \calP, central to the analysis.
$\calP$ is a graph on the odd points, consisting of vertex-disjoint paths.
In particular, a path in \calP can be an isolated vertex.
Initially, the edges in $\calP$ are precisely the base pairs $B_1$.
Eventually we will make $\calP$ a single path and then a Hamilton cycle.
Always, the edges in $\calP$ are of two types:
every hyperedge $e=(b_1,a,b_2) \in \calH$ contributes a \calP-edge $(b_1,b_2)$ (and the apexes $a$ of these will all be distinct);
and every base pair in $B$ is an edge in \calP 
(and later it will be possible to perfectly match all these base pairs
and the remaining apexes, in~$D$).
When $\calP$ is a Hamilton cycle, this ensures that there is a corresponding
loose hypergraph cycle $\calC$ as desired.

\smallskip

In this phase we will again use a set $Q$ of nodes of $D$ ``blocked'' from use.
Initially we will set $Q=Q_0$ to consist of any nodes whose inedges include any double edge; recall from \Cref{1vx} that $\card{Q_0}=o(n)$.
Initially, \calP has $\bx$ components.
Let us describe the phase just enough to count the number of nodes it will delete from each part of~$D$,
both to bound $\card Q$ and to confirm that, going in to Phase~3, $\card A = \card B$.

In Phase~2a,
each action (see~\Cref{patchA.fig}) will assign an apex in
$\AX$ to a hyperedge (along with two base points from different base pairs) and place it in~\calC,
increasing the number of paths in \calP by 1, and deleting 2 base pairs from~$B$.
Phase~2a takes $\ax$ actions
to resolve all the failed apexes, resulting in a total of
$2\ax$ deletions from the set of base pairs in $B$,
so that $\card{B_0 \bk B} = \bx + 2\ax$,
and there are $\bx + \ax$ components in~\calP.
The part $A$ is unchanged, with $\card{A_0 \bk A} = \ax$.

In Phase~2b,
each action will decrease the number of components of \calP by~1
(see~\Cref{patchB.fig}),
deleting 2 apexes from $D$ (committing them to hyperedges)
and deleting 1 base pair from~$B$.
This phase takes $\ax + \bx$ actions
to make \calP a cycle,
in the process making
$\card{A_0 \bk A} = \ax+2(\ax+\bx)$ and
$\card{B_0 \bk B} = (\bx+2\ax)+(\ax+\bx)$,
both equal to $3\ax+2\bx$.

In the end, then,
we have equal numbers of apexes and base pairs not in~$D$,
and thus also equal numbers remaining in~$D$.

The number of apex nodes deleted from $D$ by the two phases is $0+2(\ax+\bx)=2\ax+2\bx$,
and the number of base nodes deleted is $2\ax+(\ax+\bx) = 3\ax+\bx$;
by \eqref{Hcardax} each is at most $2 \gapd n$.

We can make the blocked set $Q$ arbitrarily small
by choosing $\gapd$ sufficiently small, as we now show.
For any given $\gapq>0$,
let \Cref{expansion}'s
$\eta$ equal $\gapq/k$,
and take the lemma's corresponding $\delta$ to determine $4 \gapd$
(choosing $\gapd$ smaller if required elsewhere).
Let $S$ be the set of deleted nodes.
Then, \Cref{expansion} ensures that the deleted set's inneighbourhood has size
$\card{\Nm(S)} \leq \eta \: \nn \leq (\gapq/k) n$.
Since $D$ is $k$-out,
the inneighbourhood's outneighbourhood
--- which is to say, the rest of $Q$ beyond $Q_0$ ---
has size
$$\card{Q \bk Q_0} = \card{\Np(\Nm(S))} \leq  k (\gapq/k) n = \gapq n . $$

\subsubsection{Phase 2a}

First we treat the failed apex nodes, those in $A_X$,
doing something analogous to Phase~2a for matchings.
In Phase~2b we will treat the failed bases,
which has some extra complexity.

Upon offer of odd points $b_1$ and $b_2$,
we propose a hyperedge $e=(b_1,a,b_2)$,
where $a$ is an arbitrary apex in $A_X$.
If $b_1$ and $b_2$ belong to different paths in \calP then let $b'_1 = b_1+2$ and $b'_2=b_2+2$.
If they belong to a common path $P$ in $\calP$ then,
on the induced path between them,
let $b'_1$ be the neighbour of $b_1$ and $b'_2$ that of $b_2$. 
(In the common-path case, typically $b_1' = b_1 \pm 2$ and $(b_1,b_1') \in B$, and likewise for $(b_2,b_2')$; more on this soon.)

We discard the hyperedge unless the following conditions all hold:
\begin{enumerate}[label={(H2a-C\arabic*)},ref={(H2a-C\arabic*)},leftmargin=3cm]
\item \label{s1}
$(b_1,b_2)$ is not in \calP (therefore not in~$B$).
\item \label{s2}
$(b_1,b'_1)$ and $(b_2,b'_2)$ are in $B \setminus Q$.
\item  \label{s3}
$\Nm(b_1) \cap \Nm(b_2) = \emptyset$.
\end{enumerate}
(The labelling here is H to connote Hamilton cycle, 2a the phase, and C a condition, with A for action in the next group.)
Several of these conditions are unnecessary,
but it is probable that they all hold (as will be shown),

\begin{figure}[t]
    \centering
    \begin{tikzpicture}[scale=1.3,
    every node/.style={font=\footnotesize},
    arrow/.style={    decoration={markings,mark=at position 0.5 with {\arrow{#1}}},    postaction={decorate}  }]

\draw (0,2)  -- (2,2);
\draw[dotted] (0,2) -- (3,2);
\draw  (3,2)  -- (5,2);
\foreach \x/\label in {0/$s_1$, 2/$b_1$, 3/$b_1'$, 5/$t_1$} {
  \draw (\x,2) node[circle,fill,inner sep=1pt,label=above:\label] {};
}

\draw  (0,0)  -- (2,0);
\draw[dotted] (2,0) -- (3,0);
\draw  (3,0)  -- (5,0);
\foreach \x/\label in {0/$s_2$, 2/$b_2$, 3/$b_2'$, 5/$t_2$} {
  \draw (\x,0) node[circle,fill,inner sep=1pt,label=below:\label] {};
}

\draw[dashed] (2,2) -- (2,0) node[midway,right] {$(b_1,a,b_2)$};

\end{tikzpicture}
    \caption{Salvaging a ``failed'' apex $a$ in Phase~2a.
    Upon offer of suitable odd points $b_1$ and $b_2$,
    hyperedge $(b_1,a,b_2)$ is introduced (dashed line)
    and added to $\calC$.
    The good base pairs $\set{b_1,b_1'}$ and $\set{b_2,b_2'}$ (dotted lines)
    are deleted from $D$.
    (As a failed apex, $a$ was already absent from $D$.)
    This converts two paths in $\calP$ (the figure's upper and lower lines)
    into three,
    or one path into two if $b_1$ and $b_2$ lie on a common path
    (if in the figure, $t_1=t_2$).
    }
\label{patchA.fig}
\end{figure}
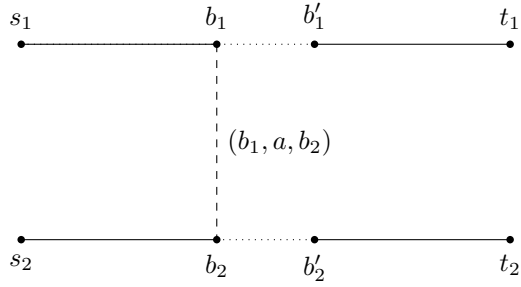

If the conditions all hold, take the following action (see \Cref{patchA.fig}).
\begin{enumerate}[label={(H2a-A\arabic*)},ref={(H2a-A\arabic*)},leftmargin=3cm]
\item
Add $(b_1,a,b_2)$ irrevocably to $\calC$, and add $(b_1,b_2)$ to \calP.
\item
Delete $(b_1,b'_1)$ and $(b_2,b'_2)$ from \calP, $B$, and~$D$.
Also, delete $a$ from $A_X$.
\item
Add $\Np(\Nm(\set{(b_1,b'_1), (b_2,b'_2)}))$ to $Q$,
blocking the outneighbours of the inneighbours of the deleted nodes.
\end{enumerate}

The action incorporates the failed apex $a$ into a hyperedge,
at the expense of removing two good base pairs $(b_1,b'_1)$ and $(b_2,b'_2)$ from $D$,
and turning two paths in \calP into three
(as shown in the figure, when $b_1$ and $b_2$ are in distinct paths)
or one into two (when $b_1$ and $b_2$ are in a common path).
Every pair in $B$ should be an edge of \calP, and since the action deletes $(b_1,b'_1)$ and $(b_2,b'_2)$ from $B$ as well as from \calP, this remains the case. 
The action decreases the outdegree of any node in $D$ by at most~1:
it is only $(b_1,b'_1)$ and $(b_2,b'_2)$ that are deleted,
by condition \ref{s2} neither was in $Q$ therefore neither has a double inedge,
and by condition \ref{s3} they have no common inneighbour.

In any trial, the conditions are likely to be satisfied.
For \ref{s1}, there is probability $O(1/n)$ that $b_1$ and $b_2$ are path neighbours.
For \ref{s2}, $b_1'$ is either $b_1+2$ or a \calP-neighbour of $b_1$, but it is likely that the \calP-neighbours of $b_1$ are just $b_1 \pm 2$, because the $(1/2-O(\gapd))n$ edges in \calP include the $(1/2-O(\gapd))n$ pairs in~$B$.
And if $b_1' = b_1 \pm 2$, it is likely that $(b_1,b_1') \in B \setminus Q$, since $B$ includes $(1/2-O(\gapd))n$ pairs of this form while $\card Q=O(\gapq)n$.
The same holds for $b_2'$.
Finally, \ref{s3} follows from \Cref{2vxs}.
Since all of the handful of failure events have probability $O(\gapd+\gapq)$,
so does their union.

\subsubsection{Phase 2b}

We now turn to the $O(\gapd)n$ failed bases, which has some extra complexity.
Details will follow but the basic idea is illustrated in \Cref{patchB.fig}.
If there were a previous semirandom offer of $(b,a)$
accepted as hyperedge $(b,a,s_1)$, with $s_1$ an endpoint of a path $P_1 \in \calP$
whose other endpoint is $t_1$,
then if there is a new offer $(b',a')$,
where $(b,b')$ is an edge in a path $P$ in $\calP$,
accepting the new offer as $(b',a',t_1)$
allows the paths $P$ and $P_1$ to be merged.
In this case, the hyperedges $(b,a,s_1)$ and $(b',a',t_1)$ are both added irrevocably to $\calC$ and $(b,b')$ is deleted from $P$, $B$, and~$D$.

The operation works equally well whether $b'$ lies between $b$ and $t$,
as shown, or between $b$ and $s$.
If \calP consists of a single path,
$s,\ldots,b,b',\ldots,t$,
then a similar operation, using hyperedges $(b,a,t)$ and $(b',a',s)$,
turns \calP into a cycle.

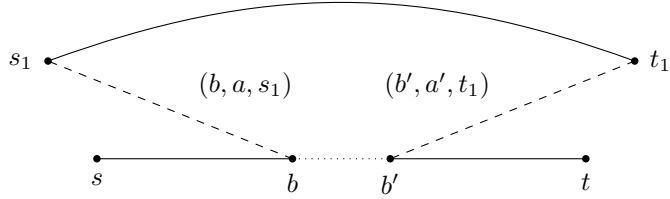
\begin{figure}[t]
    \centering
\begin{tikzpicture}[scale=1.3,
    every node/.style={font=\footnotesize},
    arrow/.style={    decoration={markings,mark=at position 0.5 with {\arrow{#1}}},    postaction={decorate}  }]

\draw  (0,0)  -- (2,0);
\draw[dotted] (2,0) -- (3,0);
\draw  (3,0)  -- (5,0);
\foreach \x/\label in {0/$s$, 2/$b$, 3/$b'$, 5/$t$} {
  \draw (\x,0) node[circle,fill,inner sep=1pt,label=below:\label] {};
}

\draw (-0.5,1) to [bend left=20] (5.5,1);
\draw (-0.5,1) node[circle,fill,inner sep=1pt,label=left:{$s_1$}]{};
\draw (5.5,1) node[circle,fill,inner sep=1pt,label=right:{$t_1$}]{};

\draw[dashed] (2,0) -- (-0.5,1) node[midway,above,xshift=1cm] {$(b,a,s_1)$};
\draw[dashed] (5.5,1) -- (3,0) node[midway,above,xshift=-1cm] {$(b',a',t_1)$};

\end{tikzpicture}
\caption{%
Patching together base-pair paths in Phase~2b.
Suppose a previous offer included an apex $a$ and base point $b$,
to which we added some path end $s_1$
to propose hyperedge $(b,a,s_1)$.
If by coincidence the current offer includes base point $b+2$
and an apex $a'$,
then, for some the opposite path end $t_1$,
propose hyperedge $(b+2,a',t_1)$.
Accepting this pair of proposals means
adding the (dashed) path edges $\set{b,s_1}$ and $\set{b+2,t_1}$
and deleting the (dotted) edge $\set{b,b+2}$,
turning two $\calP$ paths into one.
Correspondingly, we accept the two hyperedges into $\calC$,
and delete the base pair $\set{b,b+2}$ from $B$.
}
\label{patchB.fig}
\end{figure}

This merging phase proceeds in rounds, each starting with \calP having $\ell$ components and reducing that to $\floor{(9/10)\ell}$ components,
until the final round where \calP goes from a Hamilton path to a Hamilton cycle.
If any round fails, we declare failure of the whole algorithm;
we will show this to be unlikely.

We partition the $n/2$ odd points into
$\floor{n/4}$ disjoint ``cell'' pairs
$\set{1,3},\ab \set{5,7}, \set{9,11},\ab \ldots$.
(If $n \equiv 2 \pmod 4$ there is also an odd-man-out singleton $\set{n-1}$;
it will never be used, is merely a pesky detail,
and will only be mentioned once again.)
These cells are fixed for all rounds.
Hyperedges added in a round are used only within that round;
except for those used in actions,
they are ignored in future rounds and can be thought of as deleted.

\subsubsection*{Algorithm for a round}
We consider a round starting when \calP has $\ell$ components.
For each path $P_i$ in $\calP$, designate one endpoint as the start $s_i$ and the other as the terminal $t_i$.
(If $P_i$ is an isolated point, $s_i=t_i$.)
The round will consider up to $10 f$ semirandom offers, with
\begin{align}
f &= f(\ell) = n^{3/4} \ell^{1/4} .
\end{align}
The round terminates with success as soon as \calP has $\floor{(9/10)\ell}$ components,
or with failure after $10 f$ semirandom offers
or after $f^2/n$ actions have been attempted (see below), whichever comes first.

Consider only offers of the form $(b,a)$,
$b$ odd (and not the pesky odd-one-out singleton, if any) and $a$ even, discarding others.
Also discard any offer where $b$ was offered earlier in this round.
Assume then that this is the first time $b$ is offered.
If the cell partner $b'$ of $b$ also has not been offered,
complete this offer to $(b,a,t)$ where $t$ is a random terminal,
and add $(b,a,t)$ to \calH.
Consider the cell ``seeded''.
If the cell partner $b'$ was previously seen (the cell was seeded),
then some hyperedge $(b',a',t_i)$ is in \calH,
in which case complete the current offer to $(b,a,s_i)$.
Consider the cell of $b$ ``filled'',
and ``attempt action''.
Attempting action means taking action (see below)
if the following conditions are satisfied:

\begin{enumerate}[label={(H2b-C\arabic*)},ref={(H2b-C\arabic*)},leftmargin=3cm]
\item \label{m1}
The path $P_i$ with endpoints $s_i$ and $t_i$ has not already been merged
into another path: $s_i$ and $t_i$ remain path endpoints.
($P_i$ may have subsumed another path; that is fine.)
\item \label{m1x}
$(b,b') \notin P_i$. (This condition does not apply in the final round; see below.)
\item \label{m2}
$(b,b')$ is a good base pair in $B$ (and therefore an edge in \calP).
\item \label{m3}
$a$ and $a'$ are distinct apexes in $A$.
\item \label{m4}
None of $(b,b')$, $a$, nor $a'$ is in $Q$ (blocked).
\item \label{m5}
$\Nm(a) \cap \Nm(a') = \emptyset$.
\end{enumerate}

If these conditions are satisfied, the action is taken as follows.
The action decreases the number of components of \calP by~1;
see \Cref{patchB.fig}.
\begin{enumerate}[label={(H2b-A\arabic*)},ref={(H2b-A\arabic*)},leftmargin=3cm]
\item
Add the hyperedges $(b,a,s_i)$ and $(b',a',t_i)$ to $\calC$, and
add the edges $(b,s_i)$ and $(b',t_i)$ to~\calP.
\item
Delete $(b,b')$ from \calP, $B$, and $D$, and
delete $a$ and $a'$ from $A$ and $D$.
\item
Add $\Np(\Nm(\set{(b,b'),a,a'}))$ to $Q$,
blocking the outneighbours of the inneighbours of the deleted nodes.
\end{enumerate}

The last round is a special case.
Any round starting with $\ell \geq 10$ paths ends with $\floor{\frac{9}{10}\ell} \geq 9$. The rounds starting with $2\leq\ell\leq 9$ paths
each decrease the number of paths by 1,
so the last round starts with $\ell=1$.
Here, if $b$ and $b'$ share a cell in the single path $s,\ldots,b,b',\ldots,t$, then
if offered $(b,a)$ we make hyperedge $(b,a,t)$, and
if offered $(b',a')$ we make hyperedge $(b',a',s)$.
In this case, an action makes \calP a Hamilton cycle.
At that point, we proceed to Phase~3.

\subsubsection*{Analysis}

The conditions ensure that no node in $D$ has its outdegree reduced by more than~1.

Also, Phase~2b completes within $O(n)$ offers.
Round 0 starts with $\ell_0 = O(n)$ so round $i$ starts with
$\ell_i \leq \ell_0 (9/10)^i$,
and over all rounds the total number of offers is at most
\begin{align*}
\sum 10 f(\ell_i)
 &= \sum 10 n^{3/4} ((9/10)^i \ell_0)^{1/4}
 = O(n) \sum ((9/10)^{1/4})^i
 = O(n) .
\end{align*}

\medskip

It remains only to show that a.a.s., in every round, at least $(1/10)\ell$ actions are taken.

Note that for every $i$, with $f_i=f(\ell_i)$,
\begin{align}\label{fbounds}
 f_i & \geq f(1) = \Omega(n^{3/4})
 & \text{and} &&
 f_i & \leq f(\ell_0) = n \gapd^{1/4} \leq n/100 ,
\end{align}
the latter depending on having $\gapd$ sufficiently small.

We now focus on a single round and thus refer simply to $\ell$ and $f=f(\ell)$.
We will use the phrase ``overwhelmingly unlikely'' for probabilities of order $\exp(-\Omega(n^p))$ for some constant $p>0$,
and ``with overwhelming probability'' for the complement.

\begin{claim} \label{action2bAttempts}
With probability $1-\exp(-\Omega(n^{1/4}))$, $10f$ offers fill at least $f^2/n$ cells,
enabling $f^2/n$ attempted actions.
\end{claim}

\begin{proof}
Although the offers arrive successively
(and the round is terminated once $(1/10)\ell$ actions are taken),
consider the round's full $10f$ offers.
Each offer has probability asymptotically
$1/2$ of being of the requisite form $(b,a)$.
Each one of that form has equal probability of $b$ being a ``left'' or ``right''
element of a cell, i.e., of
$b \equiv 1 \pmod 4$ or $b \equiv 3 \pmod 4$.
In $10f$ offers, with probability $1-\exp(-\Omega(n))$,
there will be at least $2.4f$ each lefts and rights;
we assume henceforth that this is so.

We have $\ell \leq \ell_0 \leq O(\gapd) n$ and thus (for $\gapd$ sufficiently small)
$2.4 f(\ell) \leq n/10$ and thus each new left falls into a new cell
with probability at least $9/10$,
so the expected number of left-occupied cells is at least $2.2 f$.
Each of these left ``balls'' fell in a random cell,
and changing the location of one ball changes the number of left-occupied cells by at most~1,
so by the Azuma-Hoeffding inequality the probability of a deviation of the order of the mean is overwhelmingly unlikely:
of order
$$\exp(-\Omega(f^2/(f\cdot 1))) = \exp(-\Omega(f)) . $$
We presume henceforth that there are at least $2.1 f$ left-occupied cells.

Likewise, the right balls are overwhelmingly likely to
fall into $2.1 f$ distinct cells,
and we presume henceforth that they do.
The locations of these right-occupied cells are uniformly random.
Each has probability at least $(2.1f)/(n/4)$  of also being left-occupied,
so the expected number of filled cells is at least $17.6 f^2/n$.
Changing the location of one right-occupied cell changes the number of full cells by at most~1,
so by the Azuma-Hoeffding inequality the probability of a deviation
on the order of the mean
is overwhelmingly unlikely, of order
\begin{align*}
\exp(-\Omega( (f^2/n)^2 / (f\cdot 1)))
 &= \exp(-\Omega(f^3/n^2))
  = \exp(-\Omega(n^{1/4})) .
\end{align*}
This finishes the proof of the claim.
\end{proof}

\begin{claim} \label{action2bProb}
When a new cell is filled, the conditions for action are satisfied
with probability at least $2/10$,
assuming that the number of paths in \calP is at least $(9/10)\ell$.
\end{claim}

\begin{proof}
The location of the new cell is uniformly random over all cells,
conditioned on not being one of the cells previously filled.
For a uniformly random cell $X$,
let $F$ be the event that $X$ was not previously filled, and
let $E$ be the event that $X$ satisfies the conditions for action.
We are interested in $\Pr(E \mid F)$;
consider first just $\Pr(E)$.

For a random cell $X$, condition \ref{m1} fails w.p.\ at most $1/10$,
or the round would already have ended successfully.
Condition \ref{m1x} fails only if the randomly selected terminal
is the terminal of the path containing $b$,
and the probability of this is $1/\ell \leq 1/2$.
Conditions \ref{m2} and \ref{m3} fail with probability $O(\gapd)$,
because $A$ and $B$ contain all but an $O(\gapd)$ fraction of $A_0$ and $B_0$ respectively.
Condition \ref{m4} fails w.p.\ $O(\gapq)$ because
$\card Q = O(\gapq)n$ throughout.
Finally, \ref{m5} fails w.p.\ $O(1/n)$ by \Cref{2vxs}.
In all,
\begin{align*}
\Pr(E) & \geq 1-(1/10 +1/2 +O(\gapd+\gapq)+O(1/n)) \geq 3/10
\end{align*}
for appropriate choices of $\gapd$ and $\gapq$.

The algorithm terminates if the number of attempted actions
--- cells filled --- exceeds $f^2/n$.
Thus, the probability that a random cell $X$ was previously filled is
$\Pr(\overline{F}) \leq (f^2/n)/(n/4) = 4(f/n)^2 \leq 4/100^2$
by \eqref{fbounds}.
So, $\Pr(F) \geq 0.99$.
It follows that $\Pr(E \mid F) \geq \Pr(E \cap F) \geq \Pr(E)-\Pr(\bar F) \geq 2/10$.
\end{proof}

From \Cref{action2bAttempts}, with overwhelming probability
at least $f^2/n$ actions can be attempted, and
if the hypothesis of \Cref{action2bProb} is satisfied,
the number of these succeeding is
distributed as $B(f^2/n, 2/10)$,
and thus is overwhelmingly likely to exceed
$\frac1{10} f^2/n$:
using \eqref{fbounds} again,
the failure probability is of order
$\exp(-\Omega(f^2/n)) = \exp(-\Omega(n^{1/2}))$.

This number of successful actions would far exceed $\ell$: their ratio is
\begin{align*}
\frac{\frac1{10} f^2/n}{\ell}
 &= \frac{n^{3/2}\ell^{1/2}}{10 n \ell}
 = \frac1{10} \, (n/\ell)^{1/2}
 \geq \frac1{10} (1/\gapd)^{1/2}
 > 1 ,
\end{align*}
for $\gapd$ sufficiently small.

It is impossible that the number of successful actions exceeds $\ell$,
so we conclude that, with overwhelming probability,
the hypothesis of \Cref{action2bProb} must at some point fail.
That is, at some point $\calP$ must fall below $(9/10)\ell$,
upon which the round terminates with success.

Since the failures were all overwhelmingly unlikely
(of order $\exp(-\Omega(n^p))$ for some constant $p>0$),
by the union bound,
failure remains overwhelmingly unlikely even over the $O(\ln n)$ rounds.

Assuming success, Phase~2 terminates with \calP consisting of a single cycle.
Some of its edges correspond to hyperedges committed to \calC;
its other edges consist of base pairs comprising one part of $D$,
whose other part consists of the apexes not yet committed to \calC.
Thus $D$ is an induced subgraph of $D_1$,
in which every node's degree is at most 1 smaller than it is in $D_1$.

\subsection{Phase 3}

As previously stated, in this final phase we take a perfect matching between $A$ and $B$ in the graph $D$ with directions removed from its edges.
And, as noted earlier, the cardinalities of $A$ and $B$ match.
Lemma~\ref{Hall} implies that this matching exists, just as in Subsection~\ref{sHall}.

For each $\{b,b+2\}\in B$ and its matching partner $a\in A$, we place the hyperedge $\{b,a,b+2\}$ into \calC.
That hyperedge was formed in Phase 1, when the corresponding edge was added to $D$.
That completes the Hamilton cycle.

\subsection{Extension to \texorpdfstring{$s$}{s}-uniform hypergraphs}

The linear-time 2-offer strategy to construct a loose Hamilton cycle in a 3-uniform hypergraph extends easily to $s$-uniform hypergraphs.

To avoid excessive notation, consider $s=4$ for example.
Take base pairs $\set{1,4}, \allowbreak \set{4,7}, \allowbreak \set{7,10}, \ldots$,
and apex nodes (no longer single points) $\set{2,3},\set{5,6},\set{8,9},\ldots$.
The generalisation to other $s$ is clear.
We will see that it is easy to keep each apex node set intact, so that it behaves just like an apex point in the $s=3$ case.

In Phase~1, on offer of a base point $b$ and an apex point $a$ contained in an apex set $\aa \in A$, choose base point $b+s-1$
(to give base pair $\bb=\set{b,b+s-1}$),
choose all the other points in $\aa$ (to give apex $\aa$),
creating an edge between $\bb$ and $\aa$ in $D$
and correspondingly creating the hyperedge $\set{\bb \cup \aa}$.
This gives an auxiliary graph $D$ as before.

In Phase~2a, on offer of base points $\set{b_1,b_2}$, make hyperedge $\set{\set{b_1,b_2} \cup \aa}$ where $\aa$ is a failed apex node. As before, this uses up the node $\aa$ and turns two paths into three.

In Phase~2b, for an offered base point $b$ and apex point $a \in \aa \in A$, choose some path end point $s_1$ and the rest of $\aa$ to make candidate hyperedge $\set{\set{b,s_1}\cup \aa}$.
When another such candidate hyperedge contains $b+s-1$, just as for $s=3$, three paths can be turned into two, using these two hyperedges and deleting the base $\set{b,b+s-1}$.

With these extensions, the mechanics is exactly as for the $s=3$ case. The treatment of the auxiliary graph $D$ is identical.
This gives the strategy for constructing a loose Hamilton cycle in the 2-offer model for an $s$-uniform hypergraph.

Of course, whatever can be done in linearly many steps in the 2-offer model can also be done in the 1-offer model.
This establishes Theorem~\ref{s-uniformHc}.

\section*{Acknowledgements}
Greg is grateful
to Alan Frieze for guidance on graph Hamilton cycles,
to Olaf Parczyk for extensive discussions on semirandom hypergraphs,
and to Warach Veeranonchai for a careful reading resulting in several improvements to the manuscript.

\bibliographystyle{plain}
\bibliography{refs.bib}

\begin{thebibliography}{10}

\bibitem{behague2022}
Natalie Behague, Trent Marbach, Pawe\l{} Pra\l{}at, and Andrzej Ruci\'nski.
\newblock Subgraph games in the semi-random graph process and its
  generalization to hypergraphs.
\newblock {\em Electronic Journal of Combinatorics (accepted)}, 2024.

\bibitem{Behague_preprint}
Natalie Behague, Pawe\l{} Pra\l{}at, and Andrzej Ruci\'nski.
\newblock Creating subgraphs in semi-random hypergraph games, 2024.
\newblock Manuscript.

\bibitem{beneliezer2020fast}
Omri Ben-Eliezer, Lior Gishboliner, Dan Hefetz, and Michael Krivelevich.
\newblock Very fast construction of bounded-degree spanning graphs via the
  semi-random graph process.
\newblock {\em Proceedings of the 31st Symposium on Discrete Algorithms
  (SODA'20)}, pages 728--737, 2020.

\bibitem{beneliezer2019semirandom}
Omri Ben-Eliezer, Dan Hefetz, Gal Kronenberg, Olaf Parczyk, Clara Shikhelman,
  and Miloš Stojaković.
\newblock Semi-random graph process.
\newblock {\em Random Structures \& Algorithms}, 56(3):648--675, 2020.

\bibitem{BohmanFrieze2001}
Tom Bohman and Alan Frieze.
\newblock Avoiding a giant component.
\newblock {\em Random Structures Algorithms}, 19(1):75--85, 2001.

\bibitem{burova2022semi}
Sofiya Burova and Lyuben Lichev.
\newblock The semi-random tree process.
\newblock {\em arXiv preprint arXiv:2204.07376}, 2022.

\bibitem{devlin2017perfect}
Pat Devlin and Jeff Kahn.
\newblock Perfect fractional matchings in $k$-out hypergraphs.
\newblock {\em The Electronic Journal of Combinatorics}, 24(3):\#P3.60, 2017.

\bibitem{ER1959}
P.~Erd\H{o}s and A.~R\'{e}nyi.
\newblock On random graphs. {I}.
\newblock {\em Publ. Math. Debrecen}, 6:290--297, 1959.

\bibitem{ham_cycles_preprint}
Alan Frieze, Pu~Gao, Calum MacRury, Pawe{\l} Pra{\l}at, and Gregory~B Sorkin.
\newblock Building {H}amiltonian cycles in the semi-random graph process in
  less than $2n$ rounds.
\newblock {\em arXiv preprint}, 2023.

\bibitem{frieze2015efficient}
Alan Frieze and Gregory~B Sorkin.
\newblock Efficient algorithms for three-dimensional axial and planar random
  assignment problems.
\newblock {\em Random Structures \& Algorithms}, 46(1):160--196, 2015.

\bibitem{frieze2022hamilton}
Alan Frieze and Gregory~B Sorkin.
\newblock Hamilton cycles in a semi-random graph model.
\newblock {\em arXiv preprint arXiv:2208.00255}, 2022.

\bibitem{frieze1986maximum}
Alan~M Frieze.
\newblock Maximum matchings in a class of random graphs.
\newblock {\em Journal of Combinatorial Theory, Series B}, 40(2):196--212,
  1986.

\bibitem{gamarnik2023cliques}
David Gamarnik, Mihyun Kang, and Pawe\l{} Pra\l{}at.
\newblock Cliques, chromatic number, and independent sets in the semi-random
  process.
\newblock {\em SIAM Journal on Discrete Mathematics}, 38(3):2312--2334, 2024.

\bibitem{gao2020hamilton}
Pu~Gao, Bogumił Kamiński, Calum MacRury, and Paweł Prałat.
\newblock Hamilton cycles in the semi-random graph process.
\newblock {\em European Journal of Combinatorics}, 99:103423, 2022.

\bibitem{gao2022fully}
Pu~Gao, Calum MacRury, and Pawe\l{} Pra\l{}at.
\newblock A fully adaptive strategy for hamiltonian cycles in the semi-random
  graph process.
\newblock In Amit Chakrabarti and Chaitanya Swamy, editors, {\em Approximation,
  Randomization, and Combinatorial Optimization. Algorithms and Techniques,
  {APPROX/RANDOM} 2022, September 19-21, 2022, University of Illinois,
  Urbana-Champaign, {USA} (Virtual Conference)}, volume 245 of {\em LIPIcs},
  pages 29:1--29:22. Schloss Dagstuhl - Leibniz-Zentrum f{\"{u}}r Informatik,
  2022.

\bibitem{gao2022perfect}
Pu~Gao, Calum MacRury, and Pawe{\l} Pra{\l}at.
\newblock Perfect matchings in the semirandom graph process.
\newblock {\em SIAM Journal on Discrete Mathematics}, 36(2):1274--1290, 2022.

\bibitem{gilboa2021semi}
Shoni Gilboa and Dan Hefetz.
\newblock Semi-random process without replacement.
\newblock In {\em Extended Abstracts EuroComb 2021}, pages 129--135. Springer,
  2021.

\bibitem{JLR}
Svante Janson, Tomasz \L{}uczak, and Andrzej Ruci\'nski.
\newblock {\em Random graphs}, volume~45.
\newblock John Wiley \& Sons, 2011.

\bibitem{JKV2008}
Anders Johansson, Jeff Kahn, and Van Vu.
\newblock Factors in random graphs.
\newblock {\em Random Structures \& Algorithms}, 33(1):1--28, 2008.

\bibitem{Karonski_Frieze}
Michal Karo\'nski and Alan Frieze.
\newblock {\em Introduction to Random Graphs}.
\newblock Cambridge University Press, 2016.

\bibitem{pittel}
Michal Karoński, Ed~Overman, and Boris Pittel.
\newblock On a perfect matching in a random digraph with average out-degree
  below two.
\newblock {\em Journal of Combinatorial Theory, Series B}, 143, 03 2020.

\bibitem{macrury2022sharp}
Calum MacRury and Erlang Surya.
\newblock Sharp thresholds in adaptive random graph processes.
\newblock {\em Random Structures \& Algorithms}, 64(3):741--767, 2024.

\bibitem{McDiarmid}
Colin McDiarmid.
\newblock On the method of bounded differences.
\newblock In J.Editor Siemons, editor, {\em Surveys in Combinatorics, 1989:
  Invited Papers at the Twelfth British Combinatorial Conference}, London
  Mathematical Society Lecture Note Series, pages 148--188. Cambridge
  University Press, 1989.

\bibitem{Harjas}
Pawe{\l} Pra{\l}at and Harjas Singh.
\newblock Power of $k$ choices in the semi-random graph process.
\newblock {\em The Electronic Journal of Combinatorics}, 31(1):\#P1.11, 2024.

\bibitem{walkup1980matchings}
David~W Walkup.
\newblock Matchings in random regular bipartite digraphs.
\newblock {\em Discrete mathematics}, 31(1):59--64, 1980.

\end{thebibliography}

\end{document}